\documentclass[11pt]{amsart} 
\usepackage{amsmath}
\usepackage[retainorgcmds]{IEEEtrantools}
\usepackage{graphicx}
\usepackage[margin=1in]{geometry}
\usepackage{amsfonts}
\usepackage{amsthm}
\usepackage{setspace}
\usepackage{mathtools}
\usepackage{bbm}
\usepackage{listings}
\usepackage{mathabx}
\usepackage{amssymb}
\usepackage[mathscr]{euscript}
\usepackage{pictexwd,dcpic}
\usepackage{array}
\usepackage{enumitem}
\usepackage[final]{pdfpages}
\theoremstyle{definition}
\newtheorem{thm}{Theorem}
\newtheorem{lem}{Lemma}
\newtheorem{dfn}{Definition}

\newtheorem{exm}{Example}

\newtheorem{prp}{Proposition}
\newtheorem{cnj}{Conjecture}

\newcommand{\Z}{\mathbb{Z}}

\newcommand{\R}{\mathbb{R}}

\newcommand{\im}{\text{im}}

\author{Cole Hugelmeyer} 
\title{Diagram Systems and Generalized Finite Type Theories}

\begin{document}

\maketitle 

\begin{abstract}
We present a category theoretical generalization of the Goussarov theorem for finite type invariants, relating generating sets for generalized finite type theories with diagrams systems for the corresponding topological objects. We will demonstrate this correspondence through a few examples including the standard finite type theory and its relationship with clasp diagrams, the finite type theory of delta moves and a new diagram system called looms, and the finite type theory of combinatorial structures we call virtual transverse knots. The finite type theory of delta moves may have applications to unknotting number, and the theory of virtual transverse knots leads to many interesting and difficult conjectures. 
\end{abstract}

\section{Introduction}

In this paper, we establish a relationship between generalized finite type theories of knot-like structures and diagrammatic combinatorial systems for representing those structures. By inventing suitable new combinatorial diagram systems for representing knots and knot-like structures, we can compute the universal abelian groups of generalized finite type theories. In the existing literature, there are two prominent examples of this correspondence. Gauss diagrams correspond to the finite type theory of virtual knots, and clasp diagrams correspond to the classical finite type theory of ordinary knots \cite{virt} \cite{clasp}. In addition to building the beginnings of a generalized framework for this correspondence, we will add two new items to this list.
 
A \emph{delta move} is a transformation of knots where a strand passes through a clasp. The finite type theory of delta moves is a generalized finite type theory which is similar to the classical finite type theory of ordinary knots, except that delta moves play the role of crossing changes. Variations of the finite type theory of delta moves that we develop in this paper have been studied in the existing literature. In \cite{band}, it was shown that any delta-finite type invariant of rank $n$ which is additive under connect sum is also an ordinary finite type invariant of rank $2n$, and in \cite{doubdel}, a variation called doubled delta moves were studied, and the finite type theory was shown to not be finitely generated, as even the rank zero doubled delta move invariants fully classify knots up to $S$-equivalence. We will show that the finite type theory of delta moves relates to a diagram system we call \emph{looms}. We will prove that the universal abelian group of delta move finite type invariants of rank $n$ is finitely generated, and we will also make conjectures about the potential applications of these new invariants. In particular, looms are naturally graded by unknotting number, and this induces a natural filtration by unknotting number on the universal abelian group of rank $n$ delta move finite type invariants. If this filtration turns out to be nontrivial, these invariants could yield lower bounds for unknotting number.  
  
Second, we will introduce the notion of a \emph{virtual transverse knot} and we will construct a finite type theory for these objects using a representation system called \emph{braided Gauss diagrams}. Virtual transverse knots are puzzling because many of the simplest questions one may ask about them are extremely difficult to answer. For instance, we cannot yet distinguish any nontrivial pair of transverse knots as virtual transverse knots. One notable aspect of the finite type theory of virtual transverse knots is that the natural presentation is not finite, as there are an infinite number of braided Gauss diagrams with a given number of chords. We will give a different presentation which is finite. With a computer program, we calculate the dimensions over various finite fields for the universal vector space of virtual transverse finite type invariants up to rank 5.

\section{Diagram systems and generalized finite type theories}

We will begin by presenting a category theoretical framework that describes the relationship between diagram systems and finite type theories. We introduce the notion of the universal finite type group for a cubical complex, and then we introduce the notion of a diagram system for a cubical complex. We prove that given a diagram system, we obtain a simplified, and often finite, set of generators for the rank $n$ finite type group.

Let $[n]$ denote the set $\{1,...,n\}$. The combinatorial $n$-cube, $C_n$, is the set of all functions $[n]\to\{0,1\}$, which we call binary sequences. For a binary sequence $b$, we write $|b|$ to denote the total number of $1$s within it.  

We define the \emph{cube category}, $\mathbf{Cb}$, to be the category where the objects are the combinatorial $n$-cubes for all $n\geq 0$, and the arrows are the functions $a: C_m\to C_n$ for which there exists an ordered pair, $(f,s)$, where $f:[m]\to[n]$ injectively, $s: [n]\setminus \im(f) \to \{0,1\}$, and which has the following two properties:
\begin{itemize}
\item[1)]  For any $b\in C_m$, and any $i\in [m]$, we have $b(i) = a(b)(f(i))$.
\item[2)]  For any $b\in \im(a)$, and any $i\in [n]\setminus \im(f)$, we have $b(i) = s(i)$.
\end{itemize}

 \begin{figure}[h]
\caption{The cube map $C_2\to C_3$ with $f(1) = 1, f(2) = 3$, and $s(2) = 1$.}
\centering
\includegraphics[scale = 0.7]{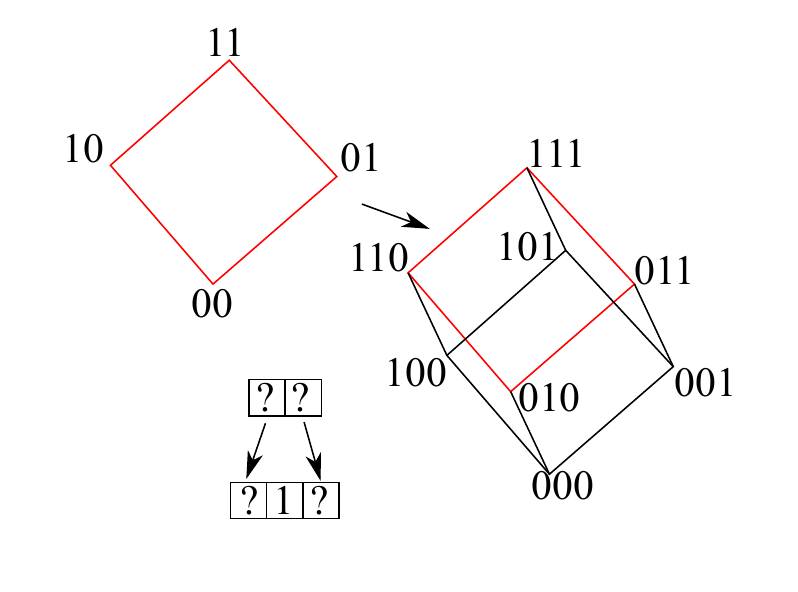}
\end{figure}

A \emph{cubical complex} is defined to be a contravariant functor $\mathbf{Cb}^{op}\to \mathbf{Set}$. Given a cubical complex $X$, we call $X(C_n)$ the set of $n$-cells of the complex, and we write it as $X_n$. 

It should be noted that this is not quite the standard definition of a cubical complex, as cells have directed edges and these directions must be preserved by glueing maps, but for our purposes this definition is the most convenient. We will not be interested in the topology of a cubical complex. Rather, we are concerned with their finite type theories. 

\begin{dfn}
Let $X$ be a cubical complex. We define the \emph{rank $n$ finite type group} of $X$ to be the abelian group $U_n(X)$ with the following presentation: 
\begin{itemize}
\item[] The generators are given by the $0$-cells of $X$.
\item[] The relations are indexed by $c\in X_{n+1}$, and are given by the following formula. $$\sum_{f\in \hom(C_0,C_{n+1})} (-1)^{|f(\varnothing)|}\cdot X(f^{op})(c)$$
\end{itemize}
Geometrically, this sum is the alternating sum over the corners of any $(n+1)$-cube. For $x\in X_0$, we will use $x$ interchangibly to denote the 0-cell, and the element of $U_n(X)$ corresponding to that 0-cell. 
\end{dfn}

If $X$ and $Y$ are cubical complexes, and $\eta$ is a natural transformation from $X$ to $Y$, then $\eta(C_0)$ is a map $X_0\to Y_0$. This induces a homomorphism $U_n(X)\to U_n(Y)$ because the relations of $U_n(X)$ are mapped to relations of $U_n(Y)$ by $\eta(C_{n+1})$. Thus, we can think of $U_n$ as a functor from the category of cubical complexes, with natural transformations as morphisms, to the category of abelian groups.

\begin{exm}
We write $K^{cc}$ to denote the cubical complex of crossing changes for knots. An element of $K_n^{cc}$ is an isotopy class of singular knots with $n$ self-intersections which are labeled $1$ through $n$, along with a choice of bijection $\sigma_i: \{0,1\}\to\{1,-1\}$ for each singularity. For a map $a:C_m \to C_n$ with corresponding ordered pair $(f,s)$, the map $K^{cc}(a^{op})$ resolves the singularities corresponding to a label $i\in [n]\setminus \im(f)$ into crossings with sign given by $\sigma_i(s(i))$, and for singularities with label $i\in \im(f)$, it replaces the label with $f^{-1}(i)$. Checking functoriality of $K^{cc}: \mathbf{Cb}^{op}\to \mathbf{Set}$ is fairly trivial, so this has been left to the reader.

We then have that $U_n(K^{cc})$ is the universal abelian group for rank $n$ finite type invariants of knots. The abelian group of rank $n$ finite type invariants of knots with coefficients in an abelian group $G$ is naturally isomorphic to the group of homomorphisms from $U_n(K^{cc})$ to $G$. The group $U_n(K^{cc})$ is the natural target space for the universal rank $n$ finite type invariant. 

It should be noted that, even if we used a fixed choice of $\sigma_i$, rather than letting it depend on the singularity, we would still get the same finite type group $U_n(K^{cc})$. The reason we include these extra bits of information is that they are necessary for us to be able to find a diagram system for the cubical complex, which is a notion we will soon define. There may be many topologically distinct cubical complexes with the same finite type theory, as when there are multiple cubes with the same corners, it will be the same as if there was only one such cube from the perspective of the abelian group quotient. 
\end{exm}

\begin{exm}
Here, we define something we will call \emph{the cubical complex of a generating set} of a monoid. We will build upon this example further in this section to help explain the concepts we introduce. To begin, let $M$ be a monoid, and let $S$ be a generating set of $M$. We then define a cubical complex $Y$ where the cubes correspond to inserting or deleting generators within a word.  More precisely, an $n$-cell of $Y_n$ consists of a function $C_n\to M$ of the form $$ b\mapsto x_0y_1^{b(\sigma(1))}x_1y_2^{b(\sigma(2))}x_2...y_n^{b(\sigma(n))}x_n$$ where $x_0,...,x_n$ are elements of $M$, $y_1,...,y_n$ are elements of $S$, and $\sigma: [n]\to[n]$ is a permutation. To define how this maps the morphisms, we say that if $f:C_m\to M$ is of the above form, and $a: C_n\to C_m$ is a cube map, then $Y(a^{op})(f) = f\circ a$.

In this case, we have that $U_n(Y)$ is isomorphic to the quotient of the monoid ring given by $\Z[M]/J^{n+1}$, where $J$ is the two-sided ideal generated by elements of the form $1-y$, for $y\in S$. To see why this is the case, we observe that when we expand the product $$x_0(1-y_1)x_1(1-y_2)x_2...(1-y_{n+1})x_{n+1}$$ we get an alternating sum of the corners of an $(n+1)$-cell of $Y$. Thus, the generators of $J^{n+1}$ coincide with the relations of $U_n(Y)$.
\end{exm}

Let $\mathbf{Inj}$ be the category of injective functions between finite sets. 

\begin{dfn}

A \emph{diagram category} is a category $D$ equipped with a functor $F_D:D\to \mathbf{Inj}$ such that the following three properties hold.

\begin{itemize}
\item[1)] $F_D$ is faithful.
\item[2)] If $a$ is an object of $D$, then for every subset $S\subseteq F_D(a)$, there is an arrow $f: b\to a$ of $C$ such that $\im(F_D(f)) = S$.
\item[3)] If $f_1: b_1\to a$ and $f_2: b_2\to a$ are such that $\im(F_D(f_1)) = \im(F_D(f_2))$, then there is an isomorphism $g: b_1\to b_2$ such that $f_1 = f_2 g$.
\end{itemize}

The objects of a diagram category are called \emph{diagrams}, and for a diagram $a$, the cardinality $|F_D(a)|$ is called the order of $a$ and is written $|a|$. This will always be a finite number because the objects of $\mathbf{Inj}$ are required to be finite sets. 

Whenever $S\subseteq F_D(a)$, we will use $a_S$ to denote the object of $D$, defined up to isomorphism, for which there exists an arrow $f:a_S\to a$ with $\im(F_D(f)) = S$. We call $a_S$ the subdiagram of $a$ corresponding to $S$. We will write $\iota_S$ to denote the arrow $a_S\to a$, which is defined up to isomorphism in the overcategory of $a$.

\end{dfn}

\begin{exm}
A Gauss diagram is a set of chords on a circle, each of which has a specified direction and a specified sign. They represent virtual knots, where the chords correspond to crossings, the directions of the chords designate over-crossings and under-crossings, and the signs of the chords designate the signs of the crossings. Gauss diagrams form a diagram category $GD$ where the arrows are subdiagram inclusions and rotational symmetries, and $F_{GD}$ takes a Gauss diagram to its set of chords.
\end{exm}

\begin{exm}
Given an alphabet, $A$, of symbols, words in that alphabet can be given the structure of a digram category which we call $W(A)$. An object of $W(A)$ is just a sequence of symbols from $A$, and $F_{W(A)}$ takes any word to the set of symbols that comprise it, where repeated symbols of the same kind are considered different elements of the set. A morphism $a: w_1\to w_2$ of $W(A)$ is a way to map $w_1$ into $w_2$ as a subword, where the letters of $w_1$ appear in order within $w_2$, but not necessarily consecutively.
\end{exm}

\begin{lem}
If $D$ is a diagram category and $a$ is an object of $D$, and if $H\subseteq S \subseteq F_D(a)$, then the arrow $\iota_H$ lifts uniquely up $\iota_S$ to give us a map $\iota_{(H,S)}: a_H\to a_S$ with $\iota_S\iota_{(H,S)} = \iota_H$.
\end{lem}

\begin{proof}
Let $Q\subseteq F_D(a_S)$ be the inverse image $Q = (F_D(\iota_S))^{-1}H$. Then we have an arrow $\iota_Q: (a_S)_Q\to a_S$, and we see that $\iota_S\iota_Q: (a_S)_Q\to a$ and has $\im(F_D(\iota_S\iota_Q)) = H = \im(F_D(\iota_H))$. Therefore, by axiom 3 of diagram categories, we have an isomorphism $g: a_H (a_S)_Q$ such that we have $\iota_H = \iota_S\iota_Qg$. We may then simply set $\iota_{(H,S)} = \iota_Qg$ to get the desired lifting. This lifting will be unique by axiom 2, as any such lifting will have $Q$ as its image under $F_D$.
\end{proof}

Given a diagram category $D$, we construct a cubical complex, denoted $D^\Box$. Intuitively speaking, a cube of this complex is given by taking all diagrams in-between a given diagram and one of its subdiagrams. More formally, we define the elements of $D^\Box_n$ to be the set of all isomorphism classes of ordered pairs $(g,\phi)$ where $g:p\to q$ is an arrow of $D$, and $\phi$ is a bijection $F_D(q) \setminus \im(F_D(g))\to [n]$. An isomorphism between two such pairs is an isomorphism of arrows between the first components that preserves the labelings from the second components. That is to say, if we have $(g_1,\phi_1)$ with $g_1: p_1\to q_1$ and $(g_2,\phi_2)$ with $g_2:p_2\to q_2$, then an isomorphism between these pairs is a pair of isomorphisms $h_p: p_1\to p_s$ and $h_q:q_1\to q_2$ with $g_2h_p = h_qg_1$ and $\phi_2F_D(h_q) = \phi_1$. Next, we need to specify how a map $a:C_m \to C_n$ acts on a pair $(g,\phi)$ with $g:p\to q$. Suppose that   the ordered pair corresponding to $a$, as in the definition of the cube category, is $(f,s)$. We then define $D^\Box (a^{op})(g,\phi)$ by the following formula. $$D^\Box (a^{op})(g,\phi)  = (\iota_{ (\im(F_D(g)) \cup \phi^{-1}s^{-1}\{1\},F_D(q) \setminus \phi^{-1}s^{-1}\{0\})}, f^{-1}\phi|_{\phi^{-1}(\im(f))} F_D(\iota_{F_D(q) \setminus \phi^{-1}s^{-1}\{0\}}))$$ This formula might be somewhat difficult to unpack, so it may not be obvious that it is functorial. A proof of functoriality is given below.

Since the isomorphism classes of objects of $D$ are naturally in bijective correspondence with the elements of $D_0^\Box$, we will usually abuse notation and write $a$ to refer to the 0-cell corresponding to the isomorphism class of $(1_a,\varnothing)$, whenever $a$ is an object of $D$.

\begin{prp}
For a diagram category $D$, the map $D^\Box: \mathbf{Cb}^{op}\to \mathbf{Set}$ is functorial.
\end{prp}

\begin{proof}
Let $a: C_m\to C_n$ with corresponding ordered pair $(f_a,s_a)$, and let $n: C_k\to C_m$ with corresponding ordered pair $(f_b,s_b)$. We wish to establish that whenever $(g,\phi)$ is as above with $g:p\to q$, we have $D^\Box((ab)^{op})(g,\phi) = D^\Box(b^{op})D^\Box(a^{op})(g,\phi)$. We will to this term by term. For the first term of the ordered pair, we see that the pair $(f_{ab},s_{ab})$ corresponding to $ab$ has $f_{ab} = f_af_b$ and $s_{ab}$ is defined piecewise as $s_a$ in $[n]\setminus \im(f_a)$ and $s_bf_a^{-1}$ in $f_a([m]\setminus \im(f_b))$. Therefore, the inverse images $s_{ab}^{-1} \{i\}$ for $i\in \{0,1\}$ are equal to $\phi^{-1}s_{a}^{-1}\{i\}\cup\phi^{-1}f_as_b^{-1}\{i\}$. Thus, we have  $$((D^\Box((ab)^{op})(g,\phi))_1 = \iota_{(\im(F_D(g))\cup\phi^{-1}s_{ab}^{-1}\{1\}, F_D(q) \setminus \phi^{-1}s_{ab}^{-1}\{0\})} $$ $$ = \iota_{(\im(F_D(g))\cup\phi^{-1}s_{a}^{-1}\{1\}\cup\phi^{-1}f_as_b^{-1}\{1\}, F_D(q) \setminus (\phi^{-1}s_{a}^{-1}\{0\}\cup\phi^{-1}f_as_b^{-1}\{0\}))} = (D^\Box(b^{op})D^\Box(a^{op})(g,\phi))_1$$

Next, we check the second coordinate. We have $$((D^\Box((ab)^{op})(g,\phi))_2 = f_{ab}^{-1}\phi|_{\phi^{-1}\im(f_{ab})} F_D(\iota_{F_D(q)\setminus \phi^{-1}s_{ab}^{-1}{0}})$$ $$ = f_b^{-1}f_a^{-1} \phi|_{\phi^{-1}\im(f_{ab})} F_D(\iota_{F_D(q)\setminus (\phi^{-1}s_{a}^{-1}\{0\}\cup\phi^{-1}f_as_b^{-1}\{0\})})$$
and $$f_a^{-1}\phi|_{\phi^{-1}\im(f_{ab}) }=\phi'|_{\phi'^{-1}\im(f_b)}$$ where $\phi' = f_a^{-1}\phi|_{\phi^{-1}(\im(f_a))} F_D(\iota_{F_D(q) \setminus \phi^{-1}s_a^{-1}\{0\}})$. 

Therefore, since $\phi' = (D^\Box(a^{op})(g,\phi))_2$, we have $$ ((D^\Box((ab)^{op})(g,\phi))_2 =   f_b^{-1}\phi'|_{\phi'^{-1}(\im(f_b))} F_D(\iota_{F_D(q)\setminus (\phi^{-1}s_{a}^{-1}\{0\}\cup\phi^{-1}f_as_b^{-1}\{0\})}) $$ $$= (D^\Box(b^{op})D^\Box(a^{op})(g,\phi))_2 $$
\end{proof}

\begin{dfn}
We define a \emph{diagram system}, $(D,\varepsilon)$, for a cubical complex $X$ to be a diagram category $D$, along with a natural transformation $\varepsilon: D^\Box\to X$, such that $\varepsilon(C_0): D^\Box_0\to X_0$ is a surjection. We write $\varepsilon_n = \varepsilon(C_n)$.
\end{dfn}

\begin{exm}
Let $M$ be a monoid with generating set $S$, and let $Y$ be the cubical complex of this generating set. We have a diagram system for $Y$ given by $(W(S),\varepsilon)$, where $\varepsilon: W(S)^\Box\to Y$ is given by letting $\varepsilon_n(g,\phi)(b)$ be obtained by taking the word that is the codomain of $g$, and removing the letters in the set $\phi^{-1}b^{-1}\{0\}$. In this way, words from the generating set form a diagram system for the cubical complex of that generating set.
\end{exm}

\begin{thm}
If $X$ is a cubical complex with a diagram system $(D,\varepsilon)$, then $U_n(X)$ is generated by elements of the form $\varepsilon_0(a)$, where $|a| \leq n$.
\end{thm}

\begin{proof}
Let $x\in X_0$. We wish to show that $x$ is equivalent modulo the rank $n$ relations to a linear combination of elements of the form $\varepsilon_0(a)$ with $|a| \leq n$. By surjectivity, we may choose an object of the diagram category, $r$, so that $\varepsilon_0(r) = x$. Let $m = |r|$. If $m \leq n$, then we are done. For the case $m>n$, we proceed by induction. Suppose that every element of the form $\varepsilon_0(b)$ with $|b| \leq m-1$ can be expressed as a linear combination of elements of the form $\varepsilon_0(a)$ with $|a| \leq n$. Then, we just need to show that $x$ can be expressed as a linear combination of elements of the form $\varepsilon_0(b)$ with $|b| \leq m-1$. To do this, let $S$ be any subset of $F_D(r)$ with $|S| = n+1$, which must exist because $m>n$. Then, let $g: q\to r$ be an arrow of the diagram category with $\im(F_D(g)) = F_D(r) \setminus S$, and let $\phi: S\to [n+1]$ be a bijection. We then have an element $c$ of $D_{n+1}^\Box$ corresponding to $(g,\phi)$. The relation corresponding to $\varepsilon_{n+1}(c)$ will then have one of its terms equal to $\varepsilon_0(r)$, and all of its other terms will be of the form $\varepsilon_0(b)$ with $|b| \leq m-1$. This gives us the desired linear combination.
\end{proof}

The idea behind this theorem is that one good way to compute the finite type theory of a cubical complex is to find a diagram system for that cubical complex. This will yield a nice set of generators. For example, in the case of the cubical complex of a generating set for a monoid and its diagram system of words, the above theorem corresponds to the fact that $\Z[M]/J^{n+1}$ is generated as an abelian group by words of length at most $n$. For more complicated cubical complexes, like $K^{cc}$, this reduction in the space of generators is extremely useful. 

We also have the following useful fact about the finite type theories of diagram categories. 

\begin{thm}[Goussarov's theorem]
Let $D$ be a diagram category, and let $Ab(D_{\leq n})$ be the free abelian group on isomorphism classes of diagrams of order at most $n$. There is an isomorphism $$s: U_n(D^\Box) \to Ab(D_{\leq n})$$ given by mapping a diagram to the sum of all subdiagrams of order at most $n$. That is to say, 
$$ s(a) = \sum_{S\subseteq F_D(a), |S| \leq n} a_S $$

Note that one isomorphism class of object may appear multiple times in this sum if it is a subobject in multiple ways. 
\end{thm}

\begin{proof}
We wish to prove that $s$ is a well-defined homomorphism of abelian groups, and that it is an isomorphism. To prove that it is a homomorphism, we need to prove that it maps relations to zero. Relations for $U_n(D^\Box)$ can be indexed by pairs $(Q,a)$ with $Q\subseteq F_D(a)$ and $|F_D(a) \setminus Q| = n+1$, and are given by the following formula. $$R_{(Q,a)} = \sum_{Q\subseteq H \subseteq F_D(a)} (-1)^{|H|} a_H$$
Applying $s$, we get $$ s(R_{(Q,a)}) = s\left(\sum_{Q\subseteq H \subseteq F_D(a)} (-1)^{|H|} a_H\right) =  \sum_{Q\subseteq H \subseteq F_D(a)} (-1)^{|H|} \sum_{S\subseteq H, |S| \leq n} a_S$$
Swapping the sums, this becomes $$ \sum_{S\subseteq F_D(a), |S| \leq n} \sum_{S\cup Q \subseteq H \subseteq F_D(a)} (-1)^{|H|}a_S $$

However, $\sum_{S\cup Q \subseteq H \subseteq F_D(a)} (-1)^{|H|}a_S$ can only be nonzero if $S\cup Q = F_D(a)$, but we know that $|F_D(a) \setminus Q| = n+1$ and $|S| \leq n$, so this is impossible. Thus, $s(R_{(Q,a)}) = 0$, so the relations map to zero. 

To prove that $s$ is an isomorphism, we construct an inverse $s_-$. We define $$ s_-(a) = \sum_{S\subseteq F_D(a)} (-1)^{|F_D(a) \setminus S|} a_S $$
Which gives us a homomorphism $Ab(D_{\leq n})\to U_n(D^\Box)$. It is now just a matter of checking that this is indeed an inverse for $s$. We will check that both $ss_-$ and $s_-s$ are identities. 

$$ss_-(a) = s\left( \sum_{H\subseteq F_D(a)} (-1)^{|F_D(a) \setminus H|} a_H \right) =  \sum_{H\subseteq F_D(a)} (-1)^{|F_D(a) \setminus H|} \sum_{S\subseteq H} a_S $$

Swapping the sums, we get

$$ ss_-(a) = \sum_{S\subseteq F_D(a)} \sum_{S\subseteq H \subseteq F_D(a)}  (-1)^{|F_D(a) \setminus H|} a_S $$ 

But $\sum_{S\subseteq H \subseteq F_D(a)}  (-1)^{|F_D(a) \setminus H|} a_S$ is zero unless $S = F_D(a)$, so  $ss_-(a) = a_{F_D(a)} = a$.\\

Finally, we check $s_-s(a)$. 

$$s_-s(a) = s_-\left(\sum_{H\subseteq F_D(a), |H| \leq n} a_H\right) = \sum_{H\subseteq F_D(a), |H| \leq n} \sum_{S\subseteq H} (-1)^{|H \setminus S|} a_S$$ 

We are working in $U_n(D^\Box)$, so we may add in terms as long as they are relations. Furthermore, if $|H|> n$, then $\sum_{S\subseteq H}(-1)^{|H\setminus S|} a_S$ is a relation. Therefore, we have $$s_-s(a) = \sum_{H\subseteq F_D(a), |H| \leq n} \sum_{S\subseteq H} (-1)^{|H \setminus S|} a_S + \sum_{H\subseteq F_D(a), |H| > n} \sum_{S\subseteq H} (-1)^{|H \setminus S|} a_S $$
$$= \sum_{H\subseteq F_D(a)} \sum_{S\subseteq H} (-1)^{|H \setminus S|} a_S $$

Swapping the sums, we get 

$$s_-s(a) = \sum_{S\subseteq F_D(a)} \sum_{S\subseteq H \subseteq F_D(a)} (-1)^{|H \setminus S|} a_S $$

But $\sum_{S\subseteq H \subseteq F_D(a)} (-1)^{|H \setminus S|} a_S$ can only be nonzero if $S = F_D(a)$, so we have $s_-s(a) = a$.
\end{proof}

Thus, when $X$ is a cubical complex, and $(D,\varepsilon)$ is a diagram system for $X$, we only need to compute the subspace $$s(\ker(U_n(\varepsilon_0)))\subseteq Ab(D_{\leq n})$$ and this will give us a presentation for $U_n(X)$, as we have an isomorphism $$ Ab(D_{\leq n})/ s(\ker(U_n(\varepsilon_0))) \simeq U_n(X)$$

This presentation can usually be computed if we have some system of equivalence moves for our diagrams for which $X_0$ is the moduli space. In the case of Gauss diagrams, these equivalence moves are just Reidemeister moves, and the relations are then just the subdiagram sums of Reidemeister moves. Things get more complicated when the equivalence moves cannot be easily stated in terms of local modifications of the diagrams, as is the case with the diagram system of looms that we will later define.

\begin{exm}\label{claspdiag}
A clasp diagram consists of the following data:
\begin{itemize}
\item[1)] A chord diagram with a finite set $C$ of unoriented chords on the circle.
\item[2)] A total ordering on $C$ called the height ordering.
\item[3)] A function $s: C\to \{1,-1\}$ called the sign function.
\end{itemize}

Clasp diagrams form a diagram category $CL$ where the morphisms are subdiagram inclusions, and $F_{CL}$ takes a diagram to its set of chords. There is a map from clasp diagrams to knots, which represents a knot by starting with a circular unknot, then adding clasps. We add the clasps along the chords of the diagram, with relative heights given by the height ordering, and clasp sign given by the sign function. This map induces a functor $k: CL^\Box\to K^{cc}$ so that $(CL,k)$ is a diagram system for $K^{cc}$. There are equivalence moves for clasp diagrams, modulo which we get the set of knots. A presentation for $U_n(K^{cc})$ can then be computed by analyzing these equivalence moves, as was done in \cite{clasp}.
\end{exm}

\begin{figure}[h]
\caption{An example of a clasp diagram and the corresponding knot.}
\centering
\includegraphics[scale = 0.8]{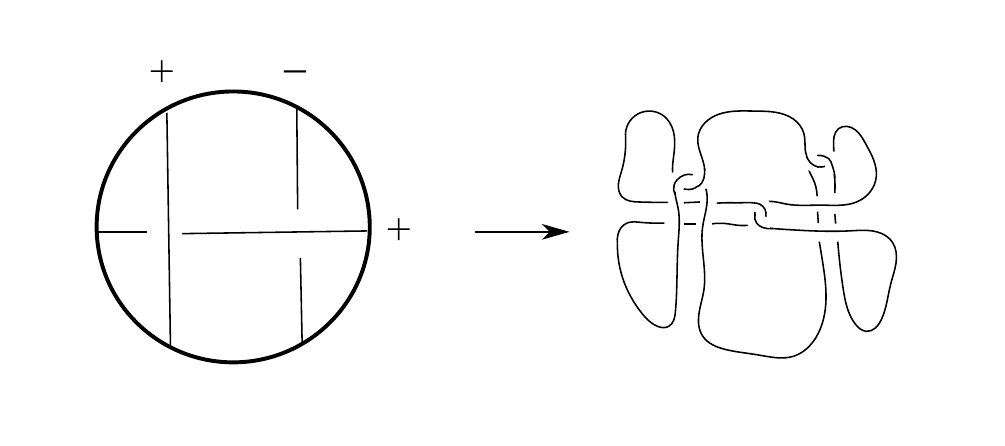}
\end{figure}

\begin{exm}
Gauss diagrams form a diagram system for the cubical complex of virtual knots, where the cubes consist of ways to switch a set of virtual crossings to real crossings. For a treatment of this the resulting finite type theory, see \cite{virt}. 
\end{exm}

Finally, it is worth noting that if we have a diagram category $D$, as well as a surjective function $\pi: D_0^\Box \to S$ from isomorphism classes of objects of $D$ to some set $S$ whose elements we wish to understand, then we can construct a cubical complex $X$ by $X_0 = S$ and $X_n =D_n^\Box$ for $n>0$, where for $a: C_0\to C_n$, we let $X(a^{op}) = \pi D^\Box(a^{op})$. We then have invariants for elements of $S$ in $U_n(X)$ for all $n$. Using Theorem 2 we can deduce that $U_n(X)$ is isomorphic to $Ab(D_{\leq n}) / R$, where $R$ is defined to be the subgroup of $Ab(D_{\leq n})$ generated by elements of the form $s(x) - s(y)$ for which $\pi(x) = \pi(y)$. If we have some system of equivalence moves on isomorphism classes of objects of $D$ for which $S$ is the moduli space, then we can find a generating set for $R$ indexed by those equivalence moves.

\section{Looms and the finite type theory of delta moves}

A \emph{loom} is a sequence of symbols from the alphabet $\{ +_i^j, -_i^j, 0_i^j, |_\pm\}_{i, j\in \Z_{>0}, \pm\in \{+,-\}}$ with the following properties.
\begin{itemize}
\item[1)] If there are $n$ symbols from $\{|_\pm\}_{ \pm\in \{+,-\}}$, then each symbol of the form $\sigma_i^j$ with $\sigma\in \{+,-,0\}$ has $1 \leq i \leq n$. 
\item[2)] If there are $m$ symbols from $\{ +_i^j, -_i^j, 0_i^j\}_{i, j\in \Z_{>0}}$, then for every $j$ from $1$ to $m$, there is exactly one symbol with superscript $j$. 
\end{itemize}

The symbols of the form $s_i^j$ with $s\in \{+,-,0\}$ are called \emph{thread symbols}, the subscript is called the \emph{target}, and the supersctipt is called the \emph{height}. The symbols $|_{\pm}$ are called \emph{bar symbols}, and their subscript is called their sign. 

Looms encode knots. To construct a knot from a loom, we start with a circular unknot and then we add $n$ vertical parallel clasps, corresponding to the bar symbols, with the appropriate signs. Then, we add in loops around these clasps, coming up from the bottom, corresponding to the thread symbols. The target determines which clasp the loop goes around, the kind of loop we add depends on the thread symbol, and the height of the loop depends on the height. The height also determines how high up along the clasp the loop goes. See the example below. 

\begin{figure}[h]
\caption{The knot corresponding to the loom $-_2^1\;+_2^4\;|_+\;0_2^3\;-_1^2\;|_-\;0_1^5$.}
\centering
\includegraphics[scale = 0.7]{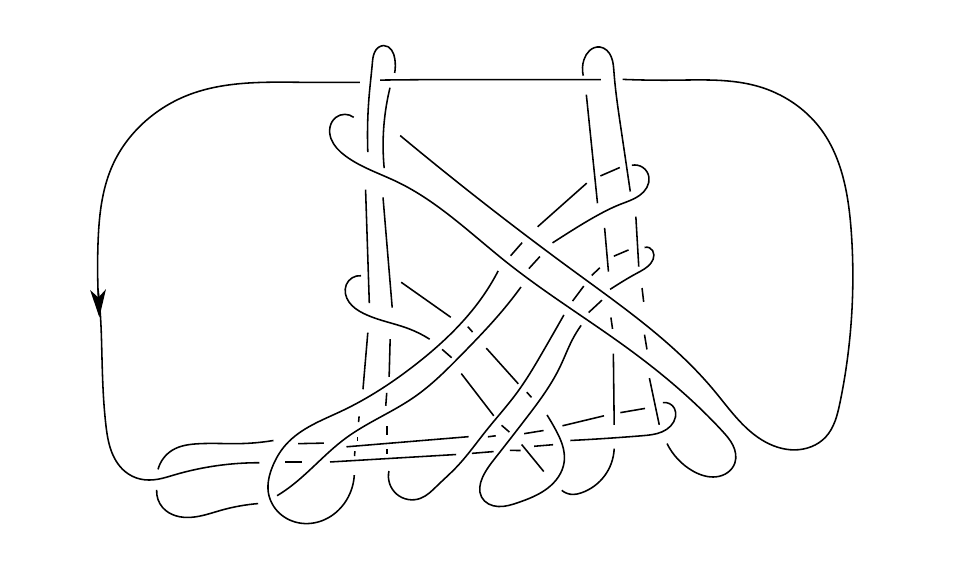}
\end{figure}

As we can see in the figure, the loops corresponding to $+$ symbols have a positive twist, the loops corresponding to the $-$ symbols have a negative twist, and the loops corresponding to $0$ symbols have no twist. The relative heights of the loops, and how high up they reach along the bars, is determined by their height numbering. 

For a loom $\ell$, we write $t(\ell)$ to denote the set of thread symbols of $\ell$. Given $s\in t(\ell)$, we can create a new loom, denoted $\ell\setminus \{s\}$, by deleting the thread symbol $s$ and decrementing all heights greater than the height of $s$ by one. Given a subset $S\subseteq t(\ell)$, we can create a loom $\ell\setminus S$ by deleting each thread symbol in $S$ and decrementing the remaining heights appropriately. 

We write $k(\ell)$ to denote the knot represented by $\ell$. 

\begin{dfn}
Given a loom, $\ell$, we define four new looms, the left and right, positive and negative stabilizations: $\ell|_-,\; \ell|_+,\; |_-\ell_{+1},\; |_+\ell_{+1}$, where $\ell_{+1}$ denotes increasing every numerical subscript by one. It is easy to check that all four of these are indeed looms. We say two looms are stabilization equivalent if one can be transformed to the other by a sequence of stabilizations and destabilizations. We write $\simeq$ to denote stabilization equivalence, and we write $LM$ to denote the set of looms up to stabilization equivalence. We give $LM$ the structure of a diagram category as follows: A morphism $a: \ell_1\to \ell_2$ is a specified set $S_a\subset t(\ell_2)$ such that $\ell_1\simeq \ell_2\setminus S_a$. The map $t$ then induces our functor $F_{LM} = t:LM\to \mathbf{Inj}$. 
\end{dfn}

\begin{figure}[h]
\caption{The stabilization $0_1^1\;|_-\;0_1^2\;\;\to \;\;0_1^1\;|_-\;0_1^2\;|_+$.}
\centering
\includegraphics[scale = 0.68]{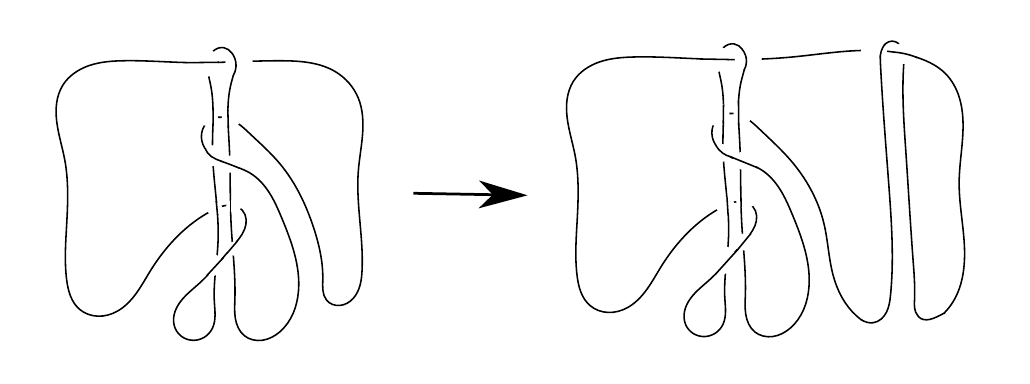}
\end{figure}

It is clear that stabilizations do not modify the resulting knot in any way, as they just add a reducible loop to the side of the knot diagram. Generally, when we talk about a ``loom'', we really mean a stabilization equivalence class of looms. We will make the distinction when it is relevant.

\begin{dfn}
A delta move is a transformation of knots where we pass a stand over a clasp. It can also be thought of as the ``forbidden'' Reidemeister move where three strands cross over each other in a cyclically symmetric way. See the figure.

\begin{figure}[h]
\caption{A delta move.}
\centering
\includegraphics[scale = 0.8]{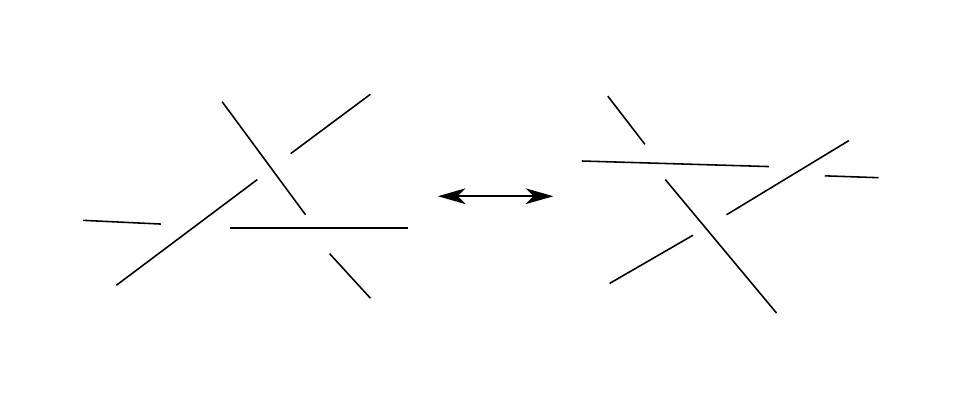}
\end{figure}
 
 We will now define $K^\Delta$, the cubical complex of delta moves. Our preferred definition for this cubical complex is not as obvious as one might think. If we just consider the 2-cells of the complex, there are two kinds of commutative behaviors for delta moves which we wish to include in the complex, depicted in the figure below. 
 
 \begin{figure}[h]
\caption{Two kinds of 2-cell.}
\centering
\includegraphics[scale = 0.8]{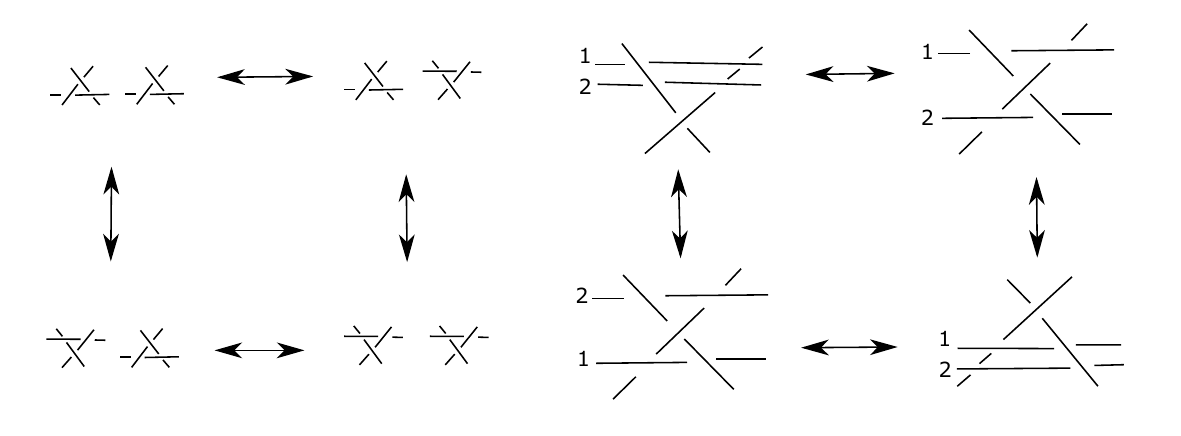}
\end{figure}
 
 To define $K^\Delta$, we first define a \emph{caterpillar band} on a knot. If we have a knot $\gamma: S^1\to \R^3$, a rank $n$ caterpillar band is a smooth embedding $\beta: ([0,1]\times[0,n+1])\to \R^3$ with the property that $\beta^{-1}(\im(\gamma)) = [0,1]\times\{0,...,n+1\}$. That is to say, it is a band which the knot passes through laterally $n$ times. See the figure below. 
 
 \begin{figure}[h]
\caption{A rank 6 caterpillar band.}
\centering
\includegraphics[scale = 0.6]{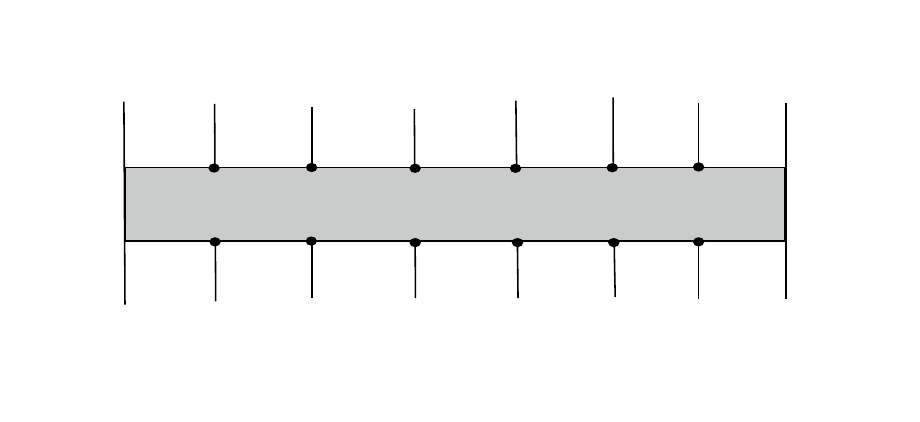}
\end{figure}

We require that the orientation of $\gamma$ is consistent with the counterclockwise orientation of the boundary of $\beta$ where it coincides with that boundary at the pair of arcs $\beta([0,1]\times\{0,n+1\})$.

We define a caterpillar knot to be a knot equipped with several disjoint caterpillar bands. The \emph{band crossings} of a caterpillar knot are the places where the knot passes through the interior of one of the caterpillar bands. We see that the total number of band crossings is equal to the sum of the ranks of the caterpillar bands. We define a \emph{labeled caterpillar knot} to be a caterpillar knot where the band crossings are labeled $1$ through $n$, and each caterpillar band has a specified sign in $\{+1,-1\}$. The total number of band crossings, $n$, is called the rank of the labeled caterpillar knot.

Finally, we define $K^\Delta_n$ to be the set of rank $n$ labeled caterpillar knots up to isotopy, such that every caterpillar band has rank at least $1$. The corners of these $n$-cells will then correspond to the knots obtained by applying clasp surgery to the bands based on their sign, and then pushing the strands that go through the band crossings off of the clasp in all $2^n$ possible ways. This gives us an $n$-cube of knots where the edges correspond to delta moves. To define this complex through our categorical language, let $a: C_m\to C_n$ be a cube map with corresponding ordered pair $(f,s)$. We define $K^\Delta(a^{op})$ to be the transformation on labeled caterpillar knots which is specified as follows.
\begin{itemize}
\item[1)] Band crossings with label $i\in [n]\setminus\im(f)$ are pushed off of their corresponding bands in the direction of their band's normal vector if $s(i) = 0$, and in the direction opposite to their band's normal vector if $s(i) = 1$. See below. (We use the standard convention that the normal vector of a surface points towards us when the surface's orientation is counterclockwise.)
 \begin{figure}[h]
\caption{Band crossings are pushed off of their band based on the value of $s$. }
\centering
\includegraphics[scale = 0.8]{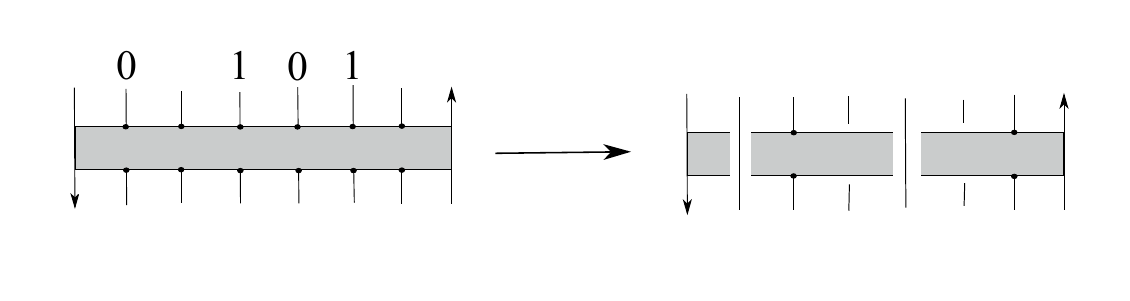}
\end{figure}
\item[2)] Band crossings with label $i\in\im(f)$ have their label replaced with $f^{-1}(i)$.
\item[3)] If after (1) and (2), there are any caterpillar bands of rank zero, replace those bands with clasps of the corresponding sign, as depicted below.
 \begin{figure}[h]
\caption{Rank zero bands get replaced with clasps where the crossings of the clasp have their sign equal to that of the band.}
\centering
\includegraphics[scale = 0.8]{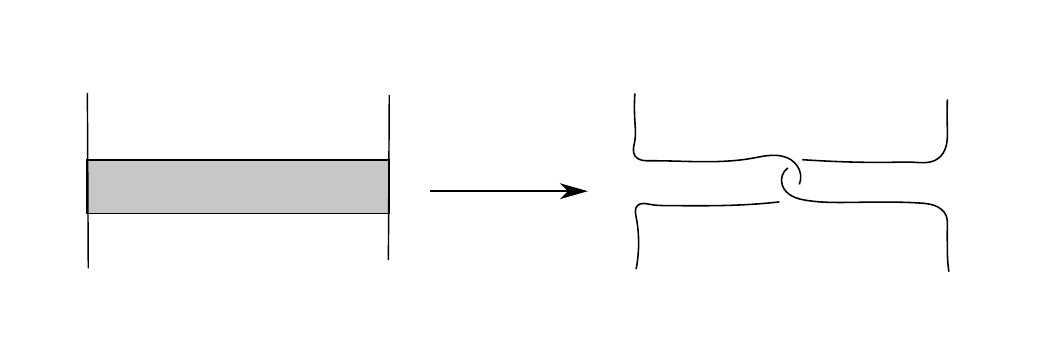}
\end{figure}
\end{itemize}

This describes our cubical complex $K^\Delta$.
\end{dfn}

\begin{lem}
The map $k$ from looms to knots can be extended to a natural transformation $$k: LM^\Box\to K^\Delta$$
\end{lem}

\begin{proof}
It is clear how $k$ acts on the 0-cells, but we need to specify how $k$ maps $LM^\Box_n\to K^\Delta_n$. To do this, let $(g,\phi)$ be a pair representing an element of $LM^\Box_n$. We have $g: \ell_1\to \ell_2$ so that $S_g$ is a set of $n$ thread symbols, and we have that $\phi: S_g\to [n]$ is a bijection. Let $B$ be the set of bar symbols of $\ell_2$ such that some thread symbol of $S_g$ has that band symbol as a target. Then, we construct a labeled caterpillar knot $k(g,\phi)$ by replacing the clasps corresponding to the bar symbols in $B$ with bands, and making the loops corresponding to the thread symbols in $S_g$ go through those bands as band crossings.

 \begin{figure}[h]
\caption{An example of $k$ applied to a cube corresponding to the loom map from the loom  $-_2^1\; |_+\;-_1^2\;|_- \;0_1^3$ to the loom $-_2^1\;+_2^4\;|_+\;0_2^3\;-_1^2\;|_-\;0_1^5$. This yields a labeled caterpillar knot with one band. }
\centering
\includegraphics[scale = 0.6]{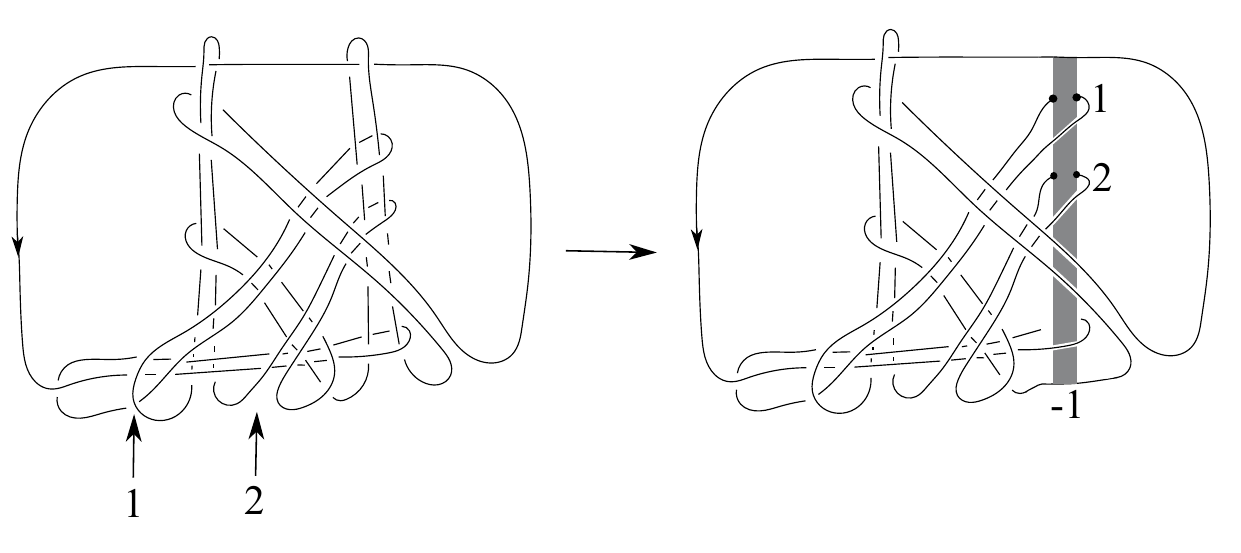}
\end{figure}

To verify the naturality of this map, let $a: C_m\to C_n$ be a cube map with corresponding ordered pair $(f,s)$. We want to show that $ kLM^\Box_n(a^{op}) = K_n^\Delta(a^{op})$. Using the definition of the cubical complex of a diagram category, we have that $LM^\Box_n(a^{op})(g,\phi)$ is equal to the isomorphism class of the following ordered pair.$$ (\iota_{ (\im (t(g)) \cup (s\phi)^{-1}\{1\},t(\ell_2) \setminus (s\phi)^{-1}\{0\})}, f^{-1}\phi|_{\phi^{-1}(\im(f))} t(\iota_{t(\ell_2) \setminus (s\phi)^{-1}\{0\}})) $$
Applying $k$ to this, we get a labeled caterpillar knot which differs from $k(g,\phi)$ in the following ways:
\begin{itemize}
\item[1)]  The loops corresponding to thread symbols in $S_g$ with label $i$ such that $s(i) = 0$ are no longer present.
\item[2)] The loops corresponding to thread symbols in $S_g$ with label $i$ such that $s(i) = 1$ no longer represent band crossings, and instead go around the corresponding band or clasp as they normally would.
\item[3)] The bands corresponding to bar symbols for which no thread symbol in $\phi^{-1}\im(f)$ has that bar symbol as a target are now clasps instead of bands.
\item[4)] Any label in the image of $f$ has been replaced by its inverse image under $f$.
\end{itemize}

We can now see that the result is exactly the labeled caterpillar knot $K_n^\Delta(a^{op})(k(g,\phi))$. Thus, $k$ is a natural transformation. 

\end{proof}

\begin{thm}
$(LM,k)$ is a diagram system for $K^\Delta$. 
\end{thm}

To prove this theorem, all we need to do is prove that every knot can be represented by a loom. The proof of this is quite complex, however, so it has been left to the end of this section. We will actually prove the following stronger theorem. 

\begin{thm}
If a knot has unknotting number $n$, then it can be represented by a loom with exactly $n$ bar symbols. 
\end{thm} 

For now, though, we will focus on using this fact to prove the following theorem. 

\begin{thm}
For all $n$, the group $U_n(K^\Delta)$ is finitely generated. 
\end{thm}

\begin{proof}[Proof of Theorem 5.]
Let $L_m$ be the full subcategory of $LM$ consisting of only those looms which are stabilization equivalent to a loom with at most $m$ bar symbols. This is a diagram category as well. From Theorem 2, we can deduce that the sequence of abelian group homomorphisms
$$ U_n(L_1^\Box)\to U_n(L_2^\Box) \to U_n(L_3^\Box) \to ... $$
is a sequence of injective maps that form a filtration for $U_n(LM^\Box)$. Furthermore, there are only finitely many stabilization equivalence classes of looms in $L_m$ with at most $n$ thread symbols. Therefore, Theorem 1 tells us that $U_n(L_m)$ is finitely generated for all $n$ and $m$. Thus, we have constructed a filtration for $U_n(LM^\Box)$ consisting of finitely generated free abelian groups. Let $A_{n,m}$ denote the image of $U_n(L_m^\Box)$ in $U_n(K^\Delta)$ under $k$. This produces a sequence of finitely generated abelian groups $A_{n,1}\subseteq A_{n,2} \subseteq A_{n,3} \subseteq... U_n(K^\Delta)$ such that $\bigcup_{m = 1}^\infty A_{n,m} = U_n(K^\Delta)$. Therefore, given any knot $x\in K^\Delta_0$, we must have that $x$ is in $A_{n,m}$ for some $m$. In fact, we claim that $x$ is in $A_{n,u(x)}$ where $u(x)$ is the unknotting number of $x$. To see why this is true, we can apply Theorem 4. Any knot $x$ can be represented by a loom with $u(x)$ bar symbols, and therefore must be in $A_{n,u(x)}$. Furthermore, Theorem 1 tells us that $U_n(K^\Delta)$ is generated by knots that can be represented by looms with at most $n$ thread symbols. Therefore, to prove that $A_{n,n} = U_n(K^\Delta)$, it suffices to prove that every knot represented by a loom with at most $n$ thread symbols has unknotting number at most $n$. This is easy to prove, because in a loom with at most $n$ thread symbols, there are at most $n$ bar symbols which are the target of some thread symbol. By undoing the clasps corresponding to those bar symbols, it completely undoes the knot, giving the unknot in at most $n$ crossing changes. Thus, we have shown that $A_{n,n} = U_n(K^\Delta)$, so $U_n(K^\Delta)$ is finitely generated. More precisely we have shown that $U_n(K^\Delta)$ is generated by knots that can be represented by a loom with at most $n$ thread symbols and at most $n$ bar symbols, of which there are finitely many. This gives a concrete, but very large, upper bound on the dimension of $U_n(K^\Delta)$.
\end{proof}

Explicitly computing the abelian groups $U_n(K^\Delta)$ is an active area of research being pursued by the author. There is no obvious system of local equivalence moves for looms, so we have to study nonlocal operations. It is difficult to find a method to calculate these abelian groups which is computationally plausible, especially since the number of looms with $n$ thread symbols and $n$ bar symbols grows \emph{extremely} quickly in $n$.  For instance, when $n=2$, the number of such looms is already $1728$. However, computing $U_n(K^\Delta)$ for small $n$ does not seem to be an intractable problem. It is a promising avenue for future research, especially considering the following observations.

In our proof that $U_n(K^\Delta)$ is finitely generated, we defined a filtration $A_{n,1}\subseteq A_{n,2}\subseteq ...$ and we proved that $A_{n,n} = U_n(K^\Delta)$. It should be noted that if this filtration is nontrivial in the sense that $A_{n,1}\neq U_n(K^\Delta)$, then $U_n(K^\Delta)$ is not generated by knots of unknotting number $1$. Given the structure of looms, it seems highly likely that the filtration will indeed be nontrivial. This would mean that delta move finite type invariants would give us lower bounds on unknotting number. We make the following conjecture.

\begin{cnj}
For some $n$, we have that $A_{n,1}\neq U_n(K^\Delta)$. Therefore, there exists a nontrivial delta move finite type invariant that vanishes on knots of unknotting number 1. Thus, delta move finite type invariants yield nontrivial lower bounds for unknotting number. 
\end{cnj}

The remainder of this section is devoted to proving Theorems 3 and 4. To give a rough sketch of the argument, we define something called a $W$-twisted loom, which is a system for representing knots that depends on a choice of \emph{comb diagram}, a kind of knot diagram like structure which is similar to a Morse link presentation. We prove that the set of knots representable by $W$-twisted looms does not depend on the choice of comb diagram $W$. Then, we prove that every knot can be represented by a $W$-twisted loom for some choice of $W$. Therefore, we know that every knot can be represented by a loom, because we can reproduce the theory of ordinary looms from the theory of $E$-twisted looms for the trivial comb diagram $E$.

\begin{dfn}
A comb diagram of rank $n$ consists of a sequence of symbols from the alphabet $$\{\subset_{i}, \supset_{i}, x_i, x_i^{-1}, |_i \}_{i\in \Z_{>0}}$$ where we think of the symbols $\subset_i$ as a creation operators, creating strands in position $i$ and $i+1$, and we think of the symbols $\supset_i$ as an annihilation operators, cobording away the strands in positions $i$ and $i+1$. We think of $x_i$ as a crossing where the strand in position $i$ crosses up over the strand in position $i+1$ as we go from left to right, and $x_i^{-1}$ as the opposite crossing. The symbols $|_i$ should be thought of as representing a band attached to the $i$-th strand, going under all the strands above it and up to infinity. We have various requirements on what constitutes a valid sequence for a comb diagram. They are as follows. 

\begin{itemize}
\item[1)] The sequence must be valid as a Morse link presentation, where we start with a single strand in position 1, and we end with a single strand in position 1. Furthermore, the resulting 1-manifold must have a single connected component. We orient it from its left endpoint to its right endpoint. 
\item[2)] There must be exactly $n$ bar symbols, and they must appear on strands that are oriented from left to right, rather than those that are oriented backwards.
\item[3)] There must be a sequence of substring modifications from the following list that reduces our word to the one consisting of $|_1$ repeated $n$ times. Here, $\varnothing$ denotes the empty sequence, and $2_{i\geq j}$ denotes $2$ if $i\geq j$ and $0$ otherwise. 
\end{itemize}
 $$ x_i\supset_i \leftrightarrow \supset_i,\;\; \subset_ix_i \leftrightarrow \subset_i,\;\; \subset_i\supset_{i+1} \leftrightarrow\varnothing,\;\; \subset_{i+1}\supset_i \leftrightarrow \varnothing $$
$$ x_ix_{i+1}x_i \leftrightarrow x_{i+1}x_ix_{i+1},\;\; x_ix_i^{-1} \leftrightarrow \varnothing,\;\; x_i^{-1}x_i \leftrightarrow \varnothing,\;\; |_{i+1}x_i \leftrightarrow x_i|_{i},\;\; |_ix_i^{-1} \leftrightarrow x_i^{-1}|_{i+1} $$
$$ x_i\supset_{i+1} \leftrightarrow x_{i+1}^{-1}\supset_i,\;\; x_i^{-1}\supset_{i+1} \leftrightarrow x_{i+1}\supset_i,\;\; \subset_{i+1}x_i \leftrightarrow \subset_ix_{i+1}^{-1},\;\; \subset_{i+1}x_i^{-1} \leftrightarrow \subset_ix_{i+1} $$
$$ x_ix_j \leftrightarrow x_jx_i \;\;\; \text{if} \;\;\; j\not\in\{i-1,i,i+1\},\;\;\;\;\;  x_i|_j \leftrightarrow |_jx_i  \;\;\; \text{if} \;\;\; j\not\in \{i,i+1\}$$
$$ x_i \subset_j \leftrightarrow \subset_{j} x_{i+ 2_{i\geq j}}\;\;\; \text{if} \;\;\; j\neq i+1,\;\;\;\;\;  \supset_j x_i \leftrightarrow x_{i + 2_{i\geq j}}\supset_j\;\;\; \text{if} \;\;\; j\neq i+1$$
$$ |_i\subset_j \leftrightarrow \subset_j |_{i + 2_{i\geq j}}, \;\;  \supset_j|_i \leftrightarrow |_{i + 2_{i\geq j}}\supset_j,\;\; \subset_{i+2}\supset_i\leftrightarrow \supset_i\subset_i \leftrightarrow \subset_i\supset_{i+2}$$
$$ \supset_i\subset_j \leftrightarrow \subset_{j+ 2_{j\geq i}} \supset_{i + 2_{i\geq j}} \;\;\; \text{if}\;\;\; i\neq j$$
Note that $|_{i+1}x_i^{-1} = x_i^{-1}|_{i}$ and $|_ix_i = x_i|_{i+1}$ are NOT valid moves, as this would require a strand to cross through the band. Furthermore, commutations like $|_i|_j = |_j|_i$ are not valid moves. The bars are essentially fixed in place. 
\end{dfn}

 \begin{figure}[h]
\caption{An example comb diagram, with word $\subset_2x^{-1}_1x^{-1}_2|_3x_2|_1\supset_2$.}
\centering
\includegraphics[scale = 0.8]{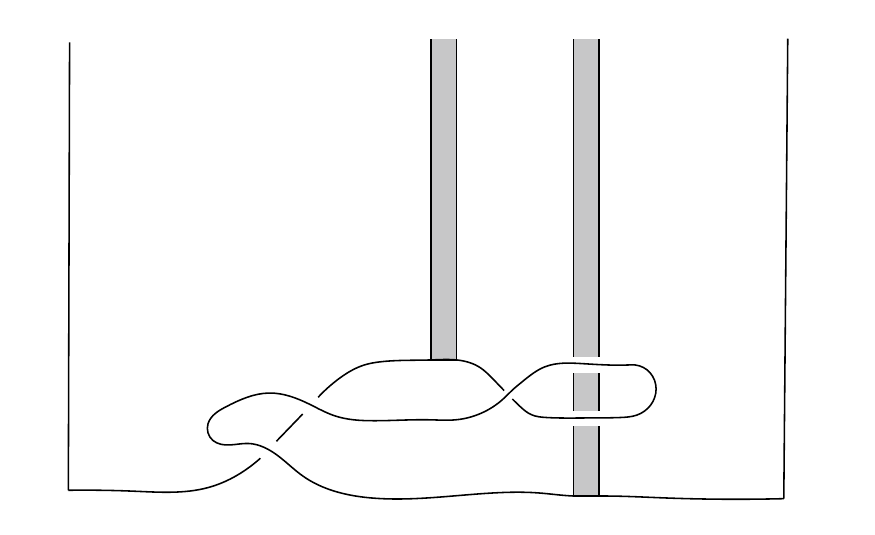}
\end{figure}

We write $E_n$ to denote the comb diagram consisting of $|_1|_1...|_1$ with $n$ bar symbols. We call this the rank $n$ trivial comb diagram. The rank of a comb diagram is the number of bar symbols that appear in it. 

\begin{dfn}
Let $W$ be a rank $n$ comb diagram. We define a $W$-twisted loom to be a sequence of symbols from the alphabet 
$$ \{\subset_{i}, \supset_{i}, x_i, x_i^{-1}, |_i , +_{i,t}^h, -_{i,t}^h, 0_{i,t}^h\}_{i,t,h\in \Z_{>0}}$$
such that the following properties hold. 
\begin{itemize}
\item[1)] If we delete all symbols from $\{+_{i,t}^h, -_{i,t}^h, 0_{i,t}^h\}_{i,t,h\in \Z_{>0}}$, then we are left with $W$. 
\item[2)] Every symbol of the form $\sigma_{i,t}^h$ with $\sigma\in \{+,-,0\}$ has $1 \leq t \leq n$. 
\item[3)] If there are $m$ symbols of the form $\sigma_{i,t}^h$ with $\sigma\in \{+,-,0\}$, then each superscript from $1$ to $m$ appears exactly once.
\end{itemize}
The symbols from $\{+_{i,t}^h, -_{i,t}^h, 0_{i,t}^h\}_{i,t,h\in \Z_{>0}}$ are called thread symbols, like with looms. The additional subscript is the strand of the comb diagram for the thread symbol, and the other two are the target and the height, as with ordinary looms. It should be noted that unlike ordinary looms, we do not have a sign associated to the clasps. 
\end{dfn}

Rather than representing knots, twisted looms represent banded unknots. To obtain a banded unknot from a $W$-twisted loom, we add clasps to the comb diagram in accordance with the thread symbols, just like with ordinary looms. These clasps always go over any strands from the comb diagram, and they link with the bands far away from the comb diagram, so as to not get tangled with it. 

 \begin{figure}[h]
\caption{A $(\subset_2x^{-1}_1x^{-1}_2|_3x_2|_1\supset_2)$-twisted loom with word $\subset_20_{2,2}^1x^{-1}_1x^{-1}_2|_3+_{2,1}^2x_2|_1\supset_2$.}
\centering
\includegraphics[scale = 0.8]{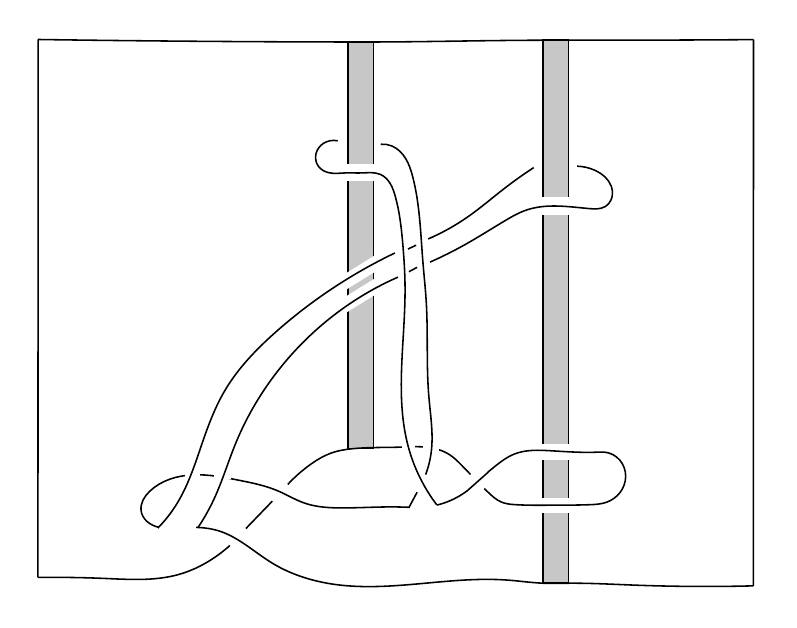}
\end{figure}

For a $W$-twisted loom $\ell$, we write $\beta(\ell)$ to denote the banded unknot represented by $\ell$. 

\begin{lem}
Let $\sigma$ and $\tau$ be symbols in $\{+,-,0\}$, and suppose that we have a $W$-twisted loom of the form $...\sigma_{i,t}^h\tau_{j,s}^w...$, then if $i< j $ and $w < h$, or $j < i$ and $h < w$, we have $$ \beta(...\sigma_{i,t}^h\tau_{j,s}^w...) = \beta(...\tau_{j,s}^w\sigma_{i,t}^h...) $$
and if if $i< j $ and $h < w$, we have $$ \beta(...\sigma_{i,t}^h\tau_{j,s}^w...) = \beta(...\tau_{j,s}^{w+2}\;0_{i,s}^{w+3}-_{i,s}^{w+1}\sigma_{i,t}^h\;0_{i,s}^w+_{i,s}^{w+4}...)^{[w+1,\infty]+4} $$ where the superscript $[w+1,\infty]+4$ means that each thread symbol in the ellipsis with height in the interval $[w+1,\infty]$ has its height incremented by 4, to prevent heights from coinciding. 

Lastly, if $j>i$ and $w>h$ we have
$$ \beta(...\sigma_{i,t}^h\tau_{j,s}^w...) = \beta(...  -_{j,t}^{h+4}0_{j,t}^h\tau_{j,s}^w +_{j,t}^{h+1}0_{j,t}^{h+3}\sigma_{i,t}^{h+2}  ...)^{[h+1,\infty] + 4}$$

Thus, if we have a sequence of thread symbols on the $i$ strand followed by a sequence thread symbols on the $j$ strand, we can transform this into a sequence of thread symbols on the $j$ strand followed by a sequence thread symbols on the $i$ strand while keeping the represented banded unknot constant.
\end{lem}

\begin{proof}
We can simply inspect the knots corresponding to the designated sequences and check that they are isotopic. In the first case, when $i< j $ and $w < h$, or $j < i$ and $h < w$, there is no conflict between the loops from the thread symbols, so they commute. In the other cases, the thread symbol on the lower strand has the smaller height, so if we were to simply commute the thread symbols, the loops would pass through each other. The additional terms compensate for this. See the isotopy depicted in figure \ref{comm}. 

 \begin{figure}[h]
\caption{\label{comm}The isotopy from $\sigma_{i,t}^h\tau_{j,s}^w$ to $\tau_{j,s}^{w+2}\;0_{i,s}^{w+3}-_{i,s}^{w+1}\sigma_{i,t}^h\;0_{i,s}^w+_{i,s}^{w+4}$ when $i<j$ and $h < w$.}
\centering
\includegraphics[scale = 0.4]{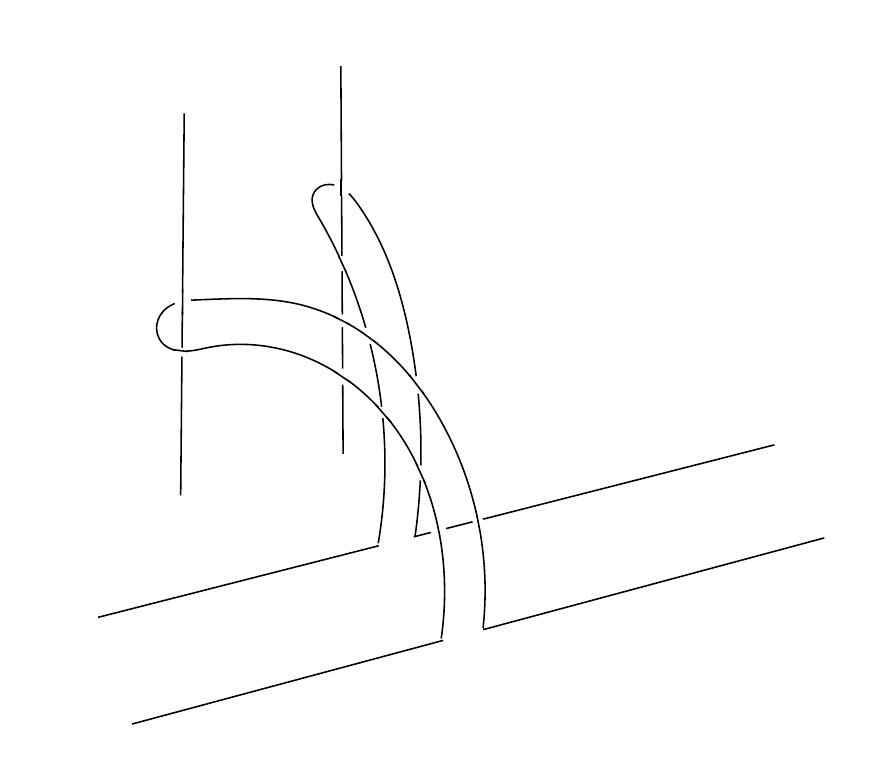}

\includegraphics[scale = 0.4]{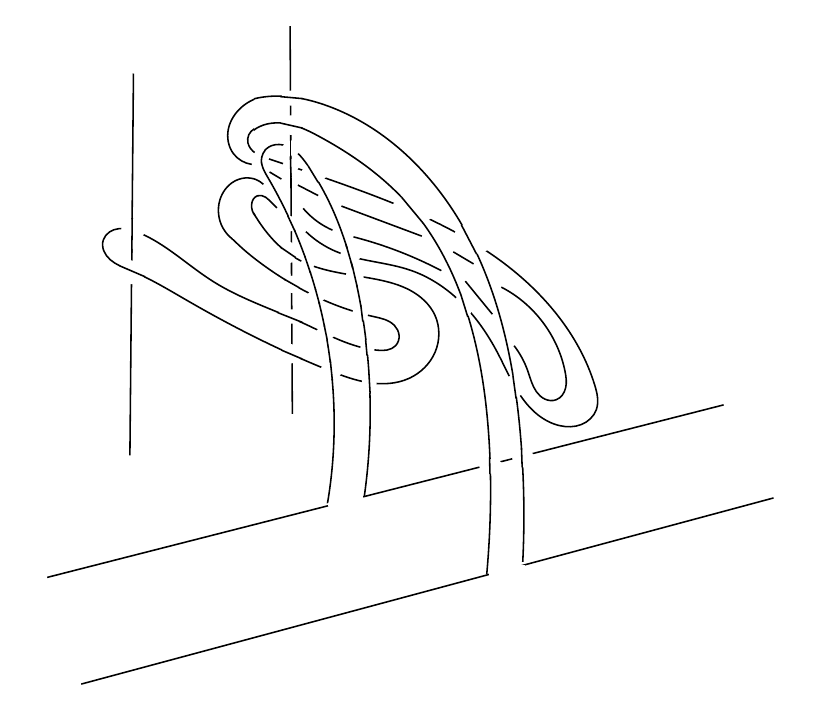}

\includegraphics[scale = 0.4]{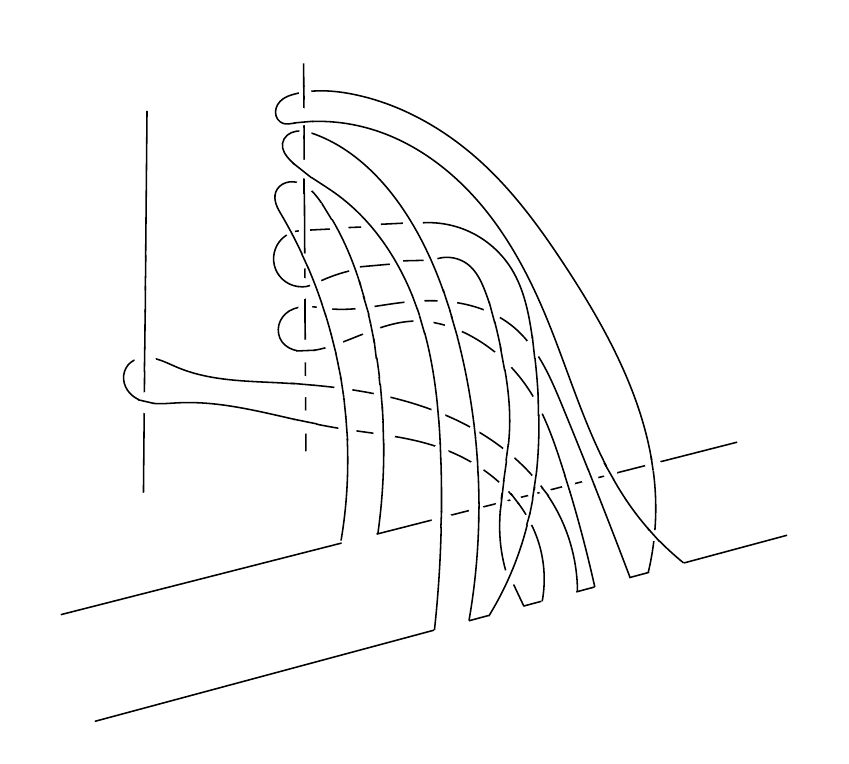}
\end{figure}

For the case when, $j>i$ and $w>h$ we have the mirror image of the isotopy depicted in the figure. 
\end{proof}

\begin{lem}
If $\ell_1$ and $\ell_2$ are $W$-twisted looms that differ by one of the following substring modifications, then $\beta(\ell_1) = \beta(\ell_2)$. Assume $\sigma$ represents some symbol in $\{+,-,0\}$.
$$ |_i\sigma_{j,t}^h \leftrightarrow \sigma_{j,t}^h|_i\;\;\; \text{if} \;\;\; i\neq j, \;\;\;\;\; x_i\sigma_{j,t}^h \leftrightarrow \sigma_{j,t}^hx_i\;\;\; \text{if} \;\;\; j\not\in \{i,i+1\}$$
$$ \supset_i \sigma_{j,t}^h \leftrightarrow \sigma_{j+2_{j\geq i},t}^h\supset_i,\;\;\;\;\; \subset_i\sigma_{j+2_{j\geq i},t}^h \leftrightarrow \sigma_{j,t}^h\subset_i$$
\end{lem}

\begin{proof}
A thread symbol on the $j$-th strand corresponds to a loop that goes over all the strands with position $>j$ and which does not cross any of the strands with position $<j$. Thus, the loop will not interact with any of the features of the comb diagram in the other strands. This implies the above commutation relations. 
\end{proof}

\begin{lem}
If $\ell_1$ and $\ell_2$ are $W$-twisted looms that differ by one of the following substring modifications, then $\beta(\ell_1) = \beta(\ell_2)$.

$$ -_{i,t}^h\supset_i \leftrightarrow 0_{i+1,t}^h \supset_i,\;\;\;\;\; 0_{i,t}^h\supset_i \leftrightarrow +_{i+1,t}^h \supset_i,\;\;\;\;\; \subset_i +_{i,t}^h \leftrightarrow \subset_i0_{i+1,t}^h  ,\;\;\;\;\; \subset_i 0_{i,t}^h \leftrightarrow \subset_i-_{i+1,t}^h$$

Furthermore, we have 

$$ \beta(...+_{i,t}^h\supset_i...) =  \beta(...  +_{i+1,t}^h0_{i+1,t}^{h+2}0_{i+1,t}^{h+1}\supset_i  ...)^{[h+1,\infty]+2}$$

$$ \beta(... \subset_i -_{i,t}^h ...) = \beta(... \subset_i 0_{i+1,t}^{h+1}0_{i+1,t}^{h+2}-_{i+1,t}^h ...)^{[h+1,\infty]+2}$$

$$ \beta(...-_{i+1,t}^h\supset_i...) =  \beta(...  0_{i,t}^{h+1}0_{i,t}^{h+2}-_{i,t}^h\supset_i  ...)^{[h+1,\infty]+2}$$

$$ \beta(... \subset_i +_{i+1,t}^h ...) = \beta(... \subset_i +_{i,t}^h0_{i,t}^{h+2}0_{i,t}^{h+1} ...)^{[h+1,\infty]+2}$$

where, as before, the superscript $[h+1,\infty]+2$ denotes incrementing heights in the designated interval by 2 to prevent any two symbols from having the same height. 

Thus, if we have a sequence of thread symbols on the $i$-th strand before a $\supset_i$ or after a $\subset_i$, we can transform it into a sequence on the $i+1$ strand, and vice-versa.
\end{lem}

\begin{proof}
It is easy to see that sliding a loop around a bend introduces a half-twist in the indicated direction. Thus, the first part of the lemma is obvious.  We then just have to show that a loop with two positive half-twists can be represented by $+_{i,t}^h0_{i,t}^{h+2}0_{i,t}^{h+1}$, and a loop with two negative half-twists can be represented by $0_{i,t}^{h+1}0_{i,t}^{h+2}-_{i,t}^h$. These are mirror images of each other, so just demonstrating one of these isotopies will suffice. The desired isotopy is depicted in the figure below. 

 \begin{figure}[h]
\caption{The isotopy from a double twisted loop to a triple of thread symbols.}
\centering
\includegraphics[scale = 0.8]{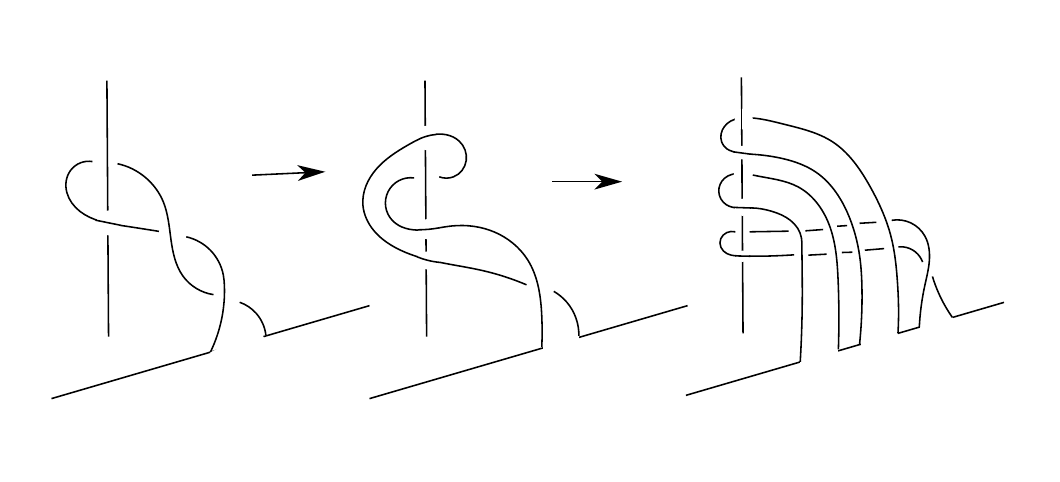}
\end{figure}
\end{proof}

\begin{lem}
If $\ell_1$ and $\ell_2$ are $W$-twisted looms that differ by one of the following substring modifications, then $\beta(\ell_1) = \beta(\ell_2)$. Assume $\sigma$ represents some symbol in $\{+,-,0\}$.
$$\sigma_{i,t}^hx_i \leftrightarrow x_i\sigma_{i+1,t}^h,\;\;\;\;\; \sigma_{i+1,t}^hx_i^{-1} \leftrightarrow x_i^{-1} \sigma_{i,t}^h$$
Furthermore, we have
$$\beta(...\sigma_{i+1,t}^hx_i...) = \beta(...x_i\sigma_{i,t}^{h+1}0_{i+1,t}^h +_{i+1,t}^{h+2}...)^{[h+1,\infty]+2} $$
$$\beta(...x_i\sigma_{i,t}^h...) = \beta(...\sigma_{i+1,t}^{h+1}0_{i,t}^{h+2}-_{i,t}^{h}x_i...)^{[h+1,\infty]+2} $$
$$\beta(...x_i^{-1}\sigma_{i+1,t}^h...) = \beta(... -_{i+1,t}^{h+2}0_{i+1,t}^{h}\sigma_{i,t}^{h+1}x_i^{-1} ...)^{[h+1,\infty]+2}$$
$$\beta(...\sigma_{i,t}^hx_i^{-1}...) = \beta(... x_i^{-1}+_{i,t}^h0_{i,t}^{h+2}\sigma_{i+1,t}^{h+1} ...)^{[h+1,\infty]+2}$$
Thus, if we have a sequence of thread symbols before an $x_i$ symbol, we can transform it into a sequence of thread symbols after that $x_i$ symbol, and vice-versa.
\end{lem}

\begin{proof}
It is easy to see that $\sigma_{i,t}^hx_i \leftrightarrow x_i\sigma_{i+1,t}^h$ and $\sigma_{i+1,t}^hx_i^{-1} \leftrightarrow x_i^{-1} \sigma_{i,t}^h$ do not change the banded unknot, since the thread symbol is just sliding over the upper strand of the crossing so it will not get tangled with the lower strand. For the remaining four equivalence moves of the lemma, it suffices to prove the first two because the last two are their mirror images. We demonstrate these moves through the following isotopies. 
 \begin{figure}[h]
\caption{The isotopies $\sigma_{i+1,t}^hx_i\leftrightarrow x_i\sigma_{i,t}^{h+1}0_{i+1,t}^h +_{i+1,t}^{h+2}$ and $x_i\sigma_{i,t}^h\leftrightarrow \sigma_{i+1,t}^{h+1}0_{i,t}^{h+2}-_{i,t}^{h}x_i$. }
\centering
\includegraphics[scale = 0.8]{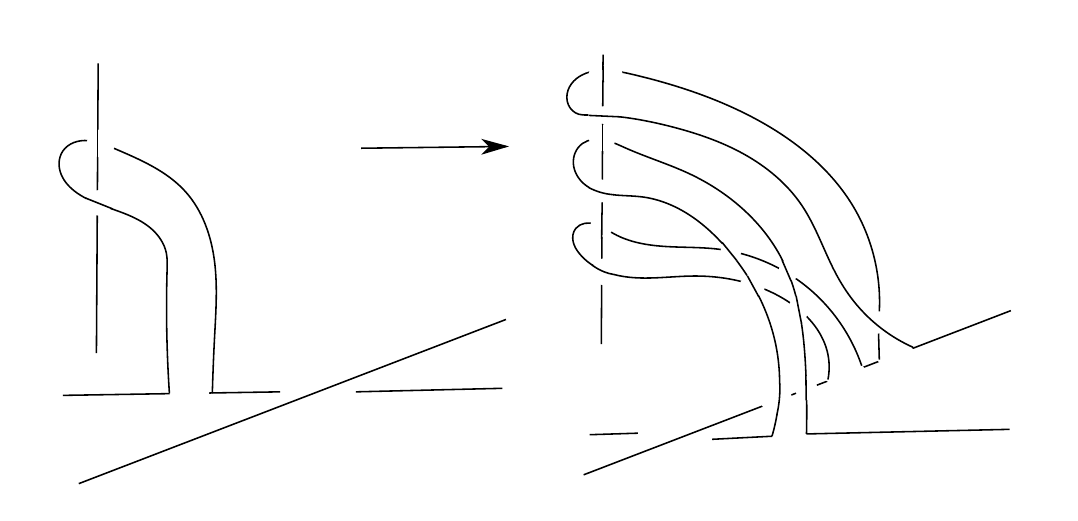}

\includegraphics[scale = 0.8]{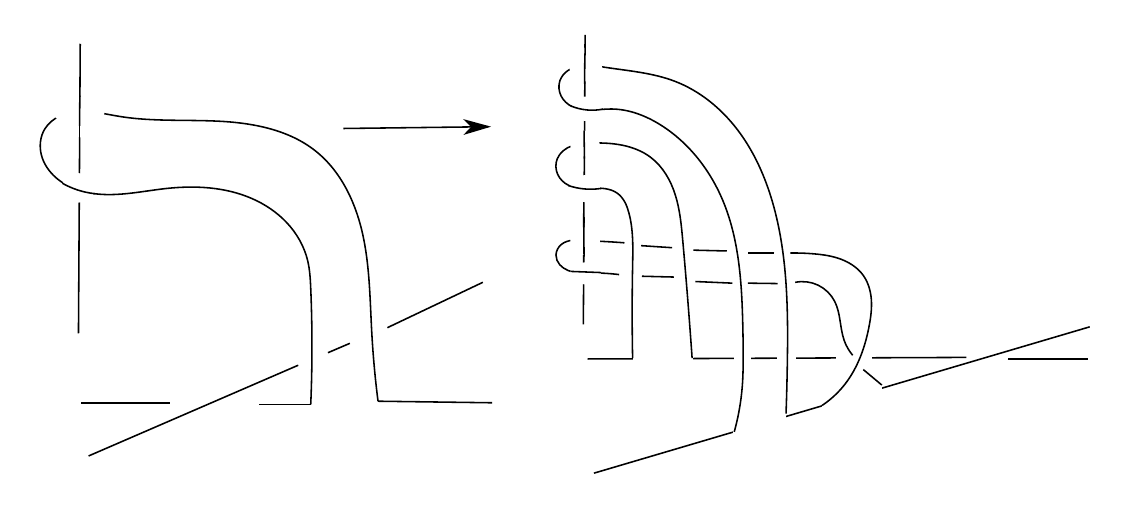}
\end{figure}
\end{proof}

We can now prove the following theorems.
\begin{thm}
The set of banded unknots representable by $W$-twisted looms is equal to the set of banded unknots representable by $E_n$-twisted looms, where $n$ is the rank of $W$. 
\end{thm}

\begin{proof}
By definition, any rank $n$ comb diagram has a sequence of equivalence moves taking it to $E_n$. Therefore, it suffices to prove that if $W$ and $W'$ are comb diagrams that differ by one of the equivalence moves, and if  a banded unknot $B$ is representable by a $W$-twisted loom $\ell$, then $B$ is also representable by a $W'$-twisted loom $\ell'$. 

We claim that there will always be a sequence of transformations of $W$-twisted looms, from the previous four lemmas, moving all the thread symbols of $\ell$ out from in between the symbols of the equivalence move that we want to apply. This would allow us to apply the equivalence move to obtain the desired $W'$-twisted loom. With the large number of different equivalence moves, this may seem like a daunting task, but it is actually quite simple. 

First, we define an \emph{arc} of our comb diagram to be a maximal interval in the corresponding 1-manifold where all the crossings are over-crossings. Intuitively, the arcs are the connected lines in the drawing of the knot diagram, when we draw under-crossings as gaps. From the previous four lemmas, it is easy to see that if we choose a point in each arc of $W$, we can apply transformations to $\ell$ to localize the thread symbols around those points, without changing the corresponding banded unknot. Therefore, the only equivalence moves of comb diagrams that we need to concern ourselves with are $x_ix_{i+1}x_i \leftrightarrow x_{i+1}x_ix_{i+1}$ and $x_ix_i^{-1}\leftrightarrow \varnothing \leftrightarrow x_i^{-1}x_i$, because these are the only two moves that trap an entire arc of the comb diagram between their symbols. However, from Lemma 6, we see that we can simply move all thread symbols out to the right of these moves by commuting them with the $x_i$ symbols. Thus, we can apply any equivalence move of comb diagrams.
\end{proof}

\begin{thm}
Let $B$ be a banded unknot with oriented bands that are combinatorially parallel. There exists a comb diagram $W$ and a $W$-twisted loom $\ell$ such that $B = \beta(\ell)$. 
\end{thm}

\begin{proof}
We claim that $B$ has a diagram as in figure \ref{bandunknot}, such that if we take all the strands that go underneath the bands and push them up through the bands, then we get an isotopically trivial banded unknot. If this is true, then we can find a twisted loom representing $B$ by switching the crossings under the bands to over-crossings, taking the comb diagram for that, and then adding $0_{i,t}^h$ thread symbols where the under crossings were, thereby obtaining a twisted loom that represents $B$. 
 \begin{figure}[h]
\caption{\label{bandunknot}A diagram for a banded unknot that can easily be interpreted as a twisted loom. }
\centering
\includegraphics[scale = 0.8]{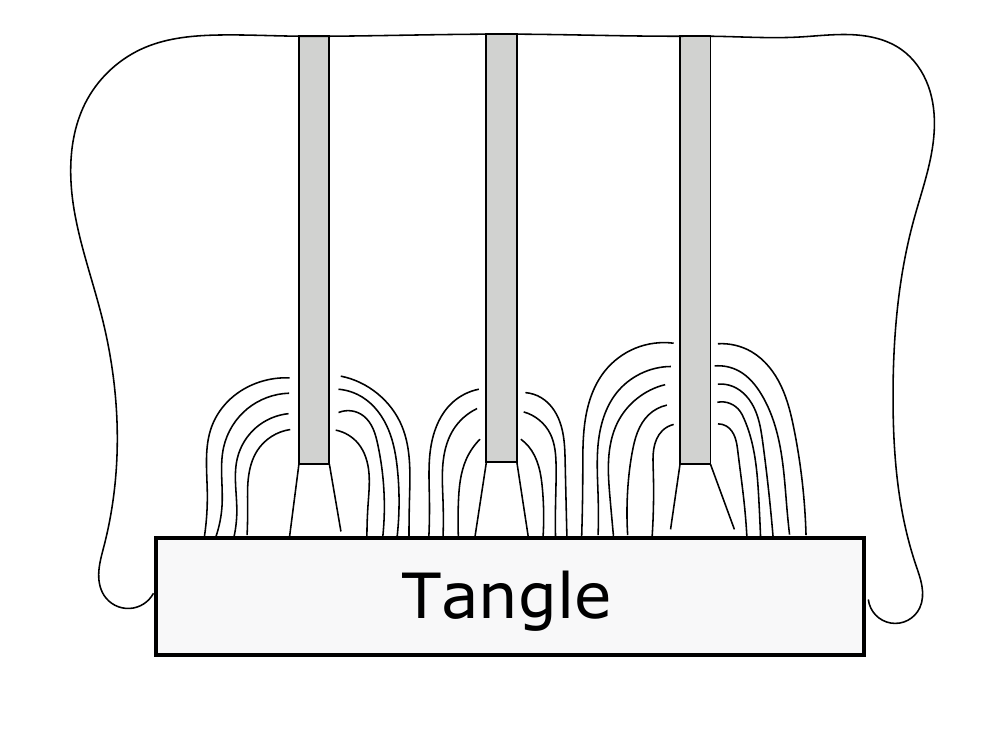}
\end{figure}

Now, we just need to prove that such a diagram exists for $B$. To do this, first take an interval on the unknot containing exactly one end of each band of $B$, and let $X$ denote the union of the interval with the bands. We have that $X$ is contractible, so for a small open neighborhood $N$ around $X$, we can isotope everything outside $N$ to be far away. We can also make it so that the strands we would need to pass through the bands to trivialize $B$ are close enough to those bands to be inside $N$, and without loss of generality, we may assume these strands go along the underside of the bands.  Thus, when we push everything outside $N$ far away, the only strands that remain close to the bands are those we wish to pass through them to trivialize $B$. This gives us a diagram in the desired form. 
\end{proof}

\begin{lem}
If a knot has unknotting number $n$, then it can be obtained from some clasp surgery on a parallel-banded unknot with $n$ bands.
\end{lem}

\begin{proof}
Let $\gamma: S^1\times[0,1]$ be a generic homotopy from our knot to the unknot that has $n$ self intersections at times $t_1,t_2,...,t_n$. Furthermore, assume this homotopy fixes a basepoint of $S^1$. Then, at each time $t_i$ we can attach a band that stays on the knot for the rest of the homotopy, encoding the clasp that we would need to add to undo the crossing change. We are free to move the endpoints of this band along the knot however we like, as long as they do not coincide with the endpoints of the bands we have already added. If, whenever a band appears, we immediately move its endpoints to be close to the basepoint, then the final result will be $n$ bands on the unknot that are combinatorially parallel, and if we do clasp surgery on these bands, we obtain the original knot. 
\end{proof}

We can finally prove Theorems $3$ and $4$.

\begin{proof}[Proof of Theorem 4.]
Given a knot $K$ of unknotting number $n$, Lemma 7 lets us represent it as a clasp surgery on a parallel banded unknot, Theorem 7 lets us represent that parallel banded unknot by a $W$-twisted loom, and Theorem 6 then lets us represent it by an $E_n$-twisted loom. Finally, we see that clasp surgery on a $E_n$-twisted loom gives us an ordinary loom with $n$ bar symbols, so we have proven that $K$ is representable by a loom with $n$ bar symbols. 
\end{proof}

\begin{proof}[Proof of Theorem 3.]
We have a natural transformation $k:LM^{\Box}\to K^\Delta$, and Theorem 4 gives us that $k_0: LM^{\Box}_0\to K^\Delta_0$ is surjective, so $(LM,k)$ is a diagram system for $K^\Delta$.
\end{proof}

\section{Virtual transverse knots and braided Gauss diagrams}

The $n$-strand virtual braid group $\text{VB}_n$ is the group generated by ``crossings'' $\sigma_1,...,\sigma_{n-1}$ and ``virtual crossings'' $v_1,...,v_{n-1}$ subject to the following relations. 
\begin{itemize}
\item[1)]  $v_i^2 = e$.
\item[2)] $s_it_{i+1}s_i = s_{i+1}t_is_{i+1}$ when we replace $(s,t)$ by $(\sigma, \sigma)$, $(v, \sigma)$, or $(v, v)$. (not $(\sigma,v)$)
\item[3)] $s_it_j = t_js_i$ where $s$ and $t$ can each be either $\sigma$ or $v$, and $|i-j|>1$.
\end{itemize}

We write $W_n$ to denote words of symbols from the set $\{\sigma_1,...,\sigma_{n-1},\sigma_1^{-1},...,\sigma_{n-1}^{-1},v_1,...,v_{n-1}\}$ and $W_n^{+}$ to denote words of symbols from the set $\{\sigma_1,...,\sigma_{n-1},v_1,...,v_{n-1}\}$. 

\begin{dfn}
A \emph{virtual transverse knot} is defined to be an equivalence class of 1-component virtual braid closures modulo positive stabilization. Two braids $\beta\in \text{VB}_n$ and $\beta'\in \text{VB}_{n+1}$ are related by positive stabilization if when $w\in W_n^1$ is a word representing $\beta$, the word $\sigma_nw\in W_{n+1}^1$ represents $\beta'$. Thus, virtual transverse knots are just virtual braids modulo conjugation and positive stabilization. 
\end{dfn}

Transverse knots can be considered ordinary braids modulo conjugation and positive stabilization, so every transverse knot gives us a virtual transverse knot. We make the following conjecture. \begin{cnj}No two transverse knots are the same virtual transverse knot.\end{cnj} This appears to be a  difficult problem. The standard proof that the map from knots to virtual knots is injective relies on the uniqueness of the fundamental quandle of a knot, which is a quite nontrivial fact of 3-manifold topology. There is no hope of adapting this proof to the transverse case, so something new is needed.  The set of virtual transverse knots will be written $VTK$, and the set of transverse knots will be written $TK$. 

There is also a relevant class of knots between virtual transverse knots and virtual knots, which we call \emph{braided virtual knots}. These are simply virtual braids modulo positive and negative stabilization. To get ordinary virtual knots, we would also need to mod out by virtual stabilizations, namely stabilizations where the added crossing is virtual. 

Let $S^1$ denote the unit complex numbers. 
\begin{dfn}
A braided Gauss diagram is defined to be a triple $(n,C,s)$ where $n$ is a positive integer, $C$ is a finite set of disjoint ordered pairs of elements of $S^1$ such that for any pair $(x,y)\in C$ we have $x^n=y^n$, and $s: C\to \{+1, -1\}$ is a function assigning a sign to each element of $C$. The elements of $C$ are called chords, and the number $n$ is called the braid index. $s$ is called the sign function.  A braided chord diagram is a braided Gauss diagram for which each chord has a positive sign. Chords are thought of as arrows going from the first term in the ordered pair to the second term. The boundary of a braided Gauss diagram is defined to be the set of points in $S^1$ which are not the endpoint of any chord. Sub-diagrams of a braided Gauss diagram are obtained by restricting to subsets of $C$. 
\end{dfn}

\begin{dfn}
Let $X = (n,C,s)$ and $X' = (n,C',s')$ be braided Gauss diagrams. We say $X$ and $X'$ are equivalent if there exists a homotopy $h: S^1\times[0,1]\to S^1$ so that
\begin{itemize}
\item[1)] $h(x,0) = x$ for all $x\in S^1$, and the map $x\mapsto h(x,t)$ is a homeomorphism for all $t\in [0,1]$.
\item[2)] For all $t\in [0,1]$ and all $(x,y)\in C$, we have $(h(x,t))^n = (h(y,t))^n$.
\item[3)] $C' = \{(h(x,1),h(y,1)): (x,y)\in C\}$, and $s'(h(x,1),h(y,1)) = s(x,y)$ for all $(x,y)\in C$. 
\end{itemize}
Let $BG$ denote the set of braided Gauss diagrams up to equivalence, and let $BG^+$ denote the set of braided chord diagrams up to equivalence. For both notations, a subscript of $n$ restricts the set to diagrams of braid index $n$. Both of these sets can be given the structure of a diagram category where arrows are subdiagram inclusions and rotational symmetries, and $F_{BG}$ takes a diagram to its set of chords. 

\end{dfn}

When we speak about braided Gauss diagrams, we will generally mean equivalence classes of braided Gauss diagrams. When we refer to braided Gauss diagrams with specified points on their boundaries, we consider them up to equivalences where the homotopy preserves the specified points in the same way it preserves the endpoints of the chords. 

It is sometimes convenient to think of braided Gauss diagrams as Gauss diagrams with a metric on their boundary such that, for each chord, the paths between the two endpoints have integer length. The metric we choose on $S^1$ for this to work is the uniform metric for which the total length of the circle is the braid index. 

There is a function $W_n^1\to BG_n$ and a function $W_n^{1+}\to BG_n^+$ for all $n$, where the 1-manifold for the virtual braid closure is mapped to the circle, and for each crossing, there is a chord from the over-crossing to the under-crossing with sign equal to the sign of the crossing.  These maps will be denoted $w\mapsto [w]$. The map $W_n\to VTK$ factors through the map $W_n^1\to BG_n$, so we may talk about braided Gauss diagrams as representing virtual transverse knots. In particular, virtual transverse knots are equivalent to braided Gauss diagrams modulo what we will call braided Reidemeister moves. The braided Reidemeister moves can be listed as follows.
\begin{itemize}
\item[1)] If $w\in W_n$ and $i$ is any index such that $\sigma_i\sigma_{i+1}\sigma_i w\in W_n^1$, then $[\sigma_i\sigma_{i+1}\sigma_i w]\leftrightarrow [\sigma_{i+1}\sigma_{i}\sigma_{i+1} w]$ is a type 1 braided Reidemeister move. 
\item[2)] If $w\in W_n^1$ and $i$ is an index, then $[w]\leftrightarrow [\sigma_i\sigma_i^{-1}w]$ and $[w]\leftrightarrow [\sigma_i^{-1}\sigma_iw]$ are type 2 braided Reidemeister moves. 
\item[3)] If $w\in W_n^1$, then $\sigma_nw\in W_{n+1}^1$, and $[w]\leftrightarrow [\sigma_nw]$ is a type 3 braided Reidemeister move.
\end{itemize}

Using these equivalence moves, we can find relations for a finite type theory of virtual transverse knots. We have a map $\pi: BG^\Box_0\to VTK$, so we are in a situation like we discussed at the end of section 2.1. We define $U_n(VTK)$ to be the abelian group generated by braided Gauss diagrams of at most $n$ chords, modulo the following relations, where a term is considered zero if it has more than $n$ chords. 

\begin{itemize}
\item[1)] For any $w \in W_n$, and any index $i$, such that $\sigma_i\sigma_{i+1}\sigma_i w\in W_n^1$, we have the relation $$ [\sigma_i\sigma_{i+1}v_iw]  + [v_i\sigma_{i+1}\sigma_iw] + [\sigma_{i}v_{i+1}\sigma_{i}w] + [\sigma_i\sigma_{i+1}\sigma_{i}w] $$  $$ -  [\sigma_{i+1}\sigma_{i}v_{i+1}w ] - [v_{i+1}\sigma_{i}\sigma_{i+1}w] - [\sigma_{i+1}v_{i}\sigma_{i+1}w] - [w\sigma_{i+1}\sigma_i\sigma_{i+1}w]    = 0$$ 
\item[2)]For any $w\in W_n^1$, and any index $i$, we have the relations $$ [w\sigma_i\sigma_i^{-1}] + [w\sigma_iv_i] + [wv_i\sigma_i^{-1}]  = 0 $$ and $$ [w\sigma_i^{-1}\sigma_i] + [w\sigma_i^{-1}v_i] + [wv_i\sigma_i]  = 0 $$
\item[3)] For any $w\in W_n^1$, we have the relation $$ [\sigma_nw] + [v_nw] - [w]  = 0 $$ where $\sigma_nw$ and $v_nw$ are regarded as elements of $W_{n+1}^1$. 
\end{itemize}

A braided Gauss diagram $x$ is then represented as an element of $U_n(VTK)$ by taking the sum of its subdiagrams of at most $n$ chords, $s(x)$. We can easily see that the relations we have defined are just $s(x)-s(y)$ when $x$ and $y$ differ by one of the braided Reidemeister moves. 

If we wish to actually compute the abelian groups $U_n(VTK)$, we run into an immediate problem. There are infinitely many braided Gauss diagrams with a given number of chords, as the braid index may be arbitrarily large. We will solve this problem by giving a presentation for $U_n(VTK)$, different from the one above, which is finite. 

\begin{dfn}
Let $CD$ be the set of ordinary chord diagrams with directed edges. We define a map $CD\to BG$ by taking a chord diagram with $n$ chords to the corresponding braided Gauss diagram of braid index $2n$ for which the distance between any two adjacent chord endpoints is exactly one, and all the chords have positive sign. We call braided Gauss diagrams \emph{unitary} when they can be obtained in this way.  Equivalently, we can consider the unitary braided Gauss diagrams to be those that can be represented in such a way that there is exactly one $2n$-th root of unity between each adjacent pair of chord endpoints, where $2n$ is the braid index and $n$ is the number of chords. 
\end{dfn}

Let $u: Ab(CD_{\leq n})\to U_n(VTK)$ be the homomorphism from the free abelian group on chord diagrams with at most $n$ chords to the finite type group of virtual transverse knots, given by mapping a chord diagram to its corresponding unitary braided Gauss diagram. 

\begin{thm}\label{surj}
The map $u$ is surjective. Thus, $U_n(VTK)$ is finitely generated. 
\end{thm}

Later in the section, we will prove this theorem, and give an explicit finite set of generators for $\ker(u)$. For now, though, we will discuss some numerical results. We wrote a computer program which uses the presentation with unitary braided Gauss diagrams to compute the vector spaces $U_n(VTK)\otimes(\Z/p\Z)$ for small $n$ and any prime $p$. 

\begin{center}
\begin{tabular}{ | m{1cm} | m{2.3cm}| m{2.3cm} | m{2.3cm} | m{2.3cm} |} 
  \hline
 $p$ & $\dim((\Z/p\Z)\otimes U_2(VTK))$ & $\dim((\Z/p\Z)\otimes U_3(VTK))$ & $\dim((\Z/p\Z)\otimes U_4(VTK))$ & $\dim((\Z/p\Z)\otimes U_5(VTK))$ \\ 
  \hline \hline
2 & 3 & 9 & 31 & 117 \\ 
  \hline
3 & 3 & 8 & 27 & 106\\
 \hline
5 & 3 & 8 & 27 & 104\\
 \hline
7 & 3 & 8 & 27 & 104\\
 \hline
\end{tabular}
\end{center}

\begin{exm}
The simplest example of two virtual transverse knots which are the same virtual knot are the unknots $[\sigma_1v_2]$ and $[v_1\sigma_2]$. These are distinguished in $(\Z/2\Z)\otimes U_2(VTK)$.
\end{exm}

We were not able to find a pair of virtual transverse knots of the same braided virtual knot type and self linking number that are distinguished by these invariants. Therefore, we make the following conjecture. 

\begin{cnj}\label{trivial}
If $x$ and $y$ are virtual transverse knots with the same braided virtual knot type and self-linking number, then they represent the same element of $U_n(VTK)$ for all $n$. 
\end{cnj}

This conjecture makes virtual transverse knots an interesting case study for the purposes of finite type theory, as they give us an example of a kind of structure that finite type invariants seem incapable of seeing. There is a long standing and very difficult question of whether universal finite type invariant of knots is a complete invariant. In order to approach this problem, we need to further develop an understanding about exactly what type of things finite type theories can see, and what type of things they cannot see. We posit that studying the finite type theory of virtual transverse knots may be helpful in understanding such questions. 

The extent to which we understand Conjecture \ref{trivial} is the following easy fact. 

\begin{prp}\label{realtriv}
If $x$ and $y$ are transverse knots with the same knot type and self-linking number, then they represent the same element of $U_n(VTK)$ for all $n$.
\end{prp}

\begin{proof}
$x$ and $y$ will become the same transverse knot if we negatively stabilize them enough times. However, it is possible to use the finite type relations to rewrite an element of $U_n(VTK)$ in terms of a linear combination of various negative stabilizations. Repeatedly doing this will allow us to rewrite both $x$ and $y$ as a linear combination of transverse knots that we know to be the same.
\end{proof}

However, we also believe the following conjecture.

\begin{cnj}
Negative stabilization is \emph{not} a unique operation on virtual transverse knots. That is to say, there are two ways to negatively stabilize some virtual transverse knot such that the two results are no longer the same virtual transverse knot.
\end{cnj}

Thus, we expect the argument in the proof of Proposition \ref{realtriv} to not apply to virtual transverse knots in general. This makes the apparent triviality of the finite type theory somewhat mysterious.

We will now work towards proving the surjectivity of $u$ and showing how to find a finite generating set for its kernel. This will comprise the remainder of this section.

\begin{dfn}
Let $U_n^+(VTK)$ be the quotient of $Ab(BG^+_{\leq n})$, the free abelian group on braided chord diagrams with at most $n$ chords, by the following relations.
\begin{itemize}
\item[1)] For any $w \in W_n^+$, and any index $i$, such that $\sigma_i\sigma_{i+1}\sigma_i w\in W_n^{1+}$, we have the relation $$ [\sigma_i\sigma_{i+1}v_iw]  + [v_i\sigma_{i+1}\sigma_iw] + [\sigma_{i}v_{i+1}\sigma_{i}w] + [\sigma_i\sigma_{i+1}\sigma_{i}w] $$  $$ -  [\sigma_{i+1}\sigma_{i}v_{i+1}w ] - [v_{i+1}\sigma_{i}\sigma_{i+1}w] - [\sigma_{i+1}v_{i}\sigma_{i+1}w] - [w\sigma_{i+1}\sigma_i\sigma_{i+1}w]    = 0$$ 
\item[2)] For any $w\in W_n^{1+}$, we have the relation $$ [\sigma_nw] + [v_nw] - [w]  = 0 $$ where $\sigma_nw$ and $v_nw$ are regarded as elements of $W_{n+1}^{1+}$. 
\end{itemize}
\end{dfn}

There is a homomorphism $\phi: U_n^+(VTK)\to U_n(VTK)$ given by the inclusion map from braided chord diagrams into braided gauss diagrams. Every relation of $U_n^+(VTK)$ is also a relation of $U_n(VTK)$, so this map is a well-defined homomorphism. 

\begin{lem}
The homomorphism $\phi$ is an isomorphism of abelian groups. 
\end{lem}

\begin{proof}
The inverse to $\phi$, which we will call $\psi$, can be obtained by replacing each negative chord with an alternating sum of adjacent positive chords. To be more precise, we apply the transformation $$[w\sigma_i^{-1}w']\mapsto \sum_{k = 1}^n(-1)^k[w(v_i\sigma_i)^kv_iw']$$ repeatedly until there are no more negative chords. The type 2 relations for $U_n(VTK)$ map to zero under this transformation, and the other relations are preserved, so it is a well defined homomorphism. It is immediate that $\psi\phi$ is the identity, and $\phi\psi$ is the identity because $[w\sigma_i^{-1}w']$ and $\sum_{k = 1}^n(-1)^k[w(v_i\sigma_i)^kv_iw']$ are always equivalent modulo type 2 relations.
\end{proof}

\begin{dfn}
A \emph{numbered chord diagram} is a chord diagram with directed chords, equipped with a choice of nonnegative integer for each section of the boundary of the diagram between chord endpoints. We define $U_n(NCD)$ to be the group generated by numbered chord diagrams of at most $n$ chords, subject to the three types of relations depicted in figures \ref{type1}, \ref{type2}, and \ref{type3}.
 \begin{figure}[h]
\caption{\label{type1}Relations of type 1. The mirror image of the depicted relation is also a relation.}
\centering
\includegraphics[scale = 0.5]{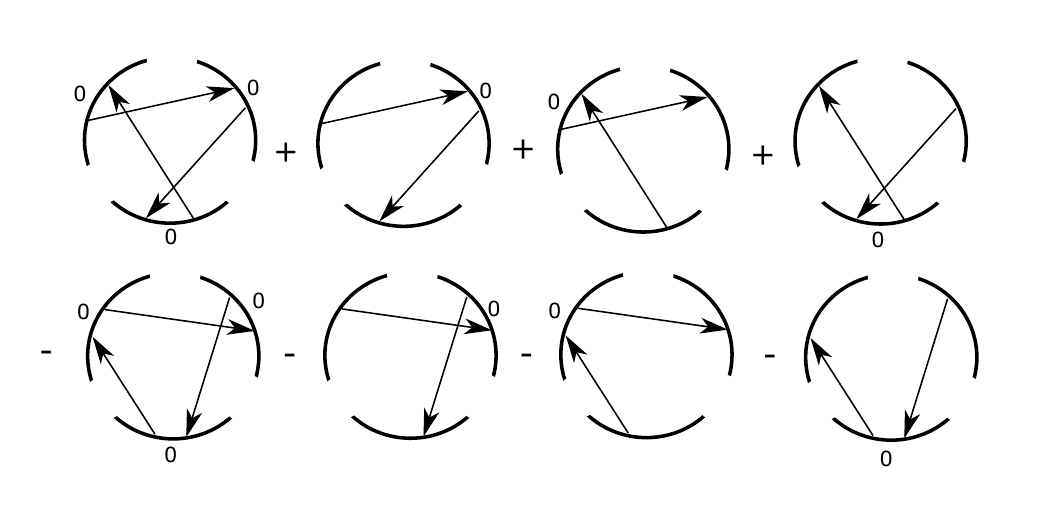}
\end{figure}
 \begin{figure}[h]
\caption{\label{type2}Relations of type 2. Reversing the depicted chord does NOT yield a valid relation.}
\centering
\includegraphics[scale = 0.5]{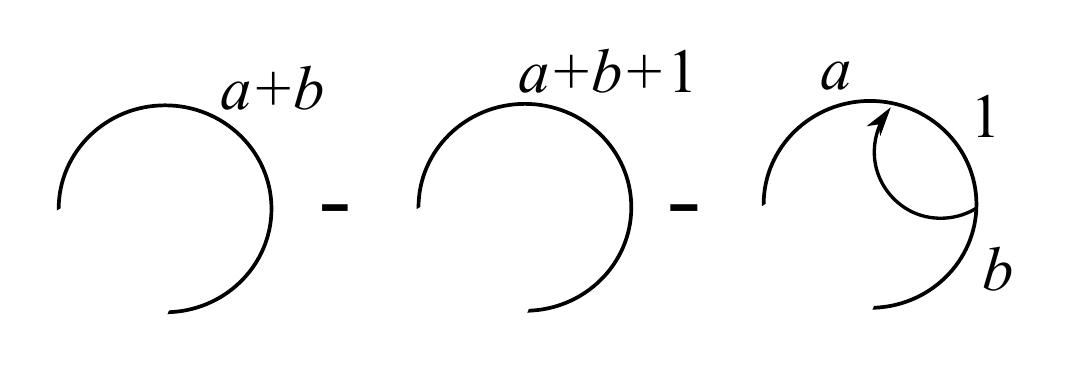}
\end{figure}
 \begin{figure}[h]
\caption{\label{type3}Relations of type 3. In the case where the $a$ and $c$ section, or the $b$ and $d$ sections are the same section of the diagram, the number in that section does not change between the two terms, as in both terms that section gets a single $+1$.}
\centering
\includegraphics[scale = 0.5]{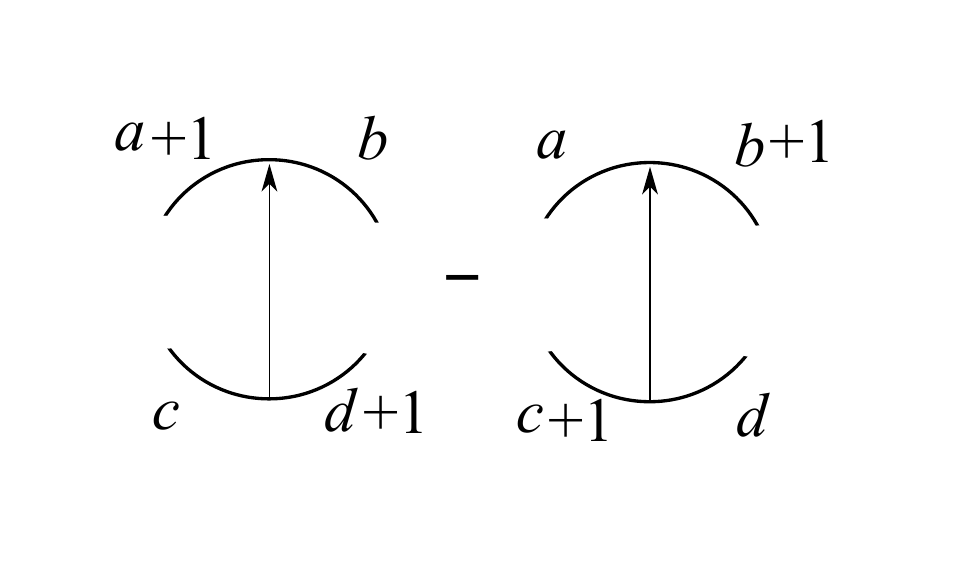}
\end{figure}
\end{dfn}

We can construct a homomorphism $w: U_n^+(VTK)\to U_n(NCD)$ by taking a braided chord diagram $X$, and then choosing a representative of its equivalence class where the $m$-th roots of unity in $S^1$ never coincide with chord endpoints, where $m$ is the braid index. Then, the numbers for the corresponding numbered chord diagram are the number of roots of unity that lie in each of the sections of the boundary of the diagram. Although this numbered chord diagram will not be uniquely determined by the equivalence class of $X$, it will be determined up to type 3 relations in $U_n(NCD)$. This is because whenever we move a chord endpoint past an $m$-th root of unity, we also move the other endpoint of that chord past an $m$-th root of unity, and this corresponds to applying a type 3 relation to the resulting numbered chord diagram. We then see that $w: U_n^+(VTK)\to U_n(NCD)$ is a well-defined homomorphism because it maps type 1 relations in $U_n^+(VTK)$ to type 1 relations in $U_n(NCD)$, and similarly with type 2 relations.

Let $Ab(NCD_{\leq n})$ be the free abelian group on numbered chord diagrams of at most $n$ chords, and let $Ab(CD_{\leq n})$ be the free abelian group on chord diagrams of at most $n$ chords. We will define a map $v: Ab(NCD_{\leq n})\to Ab(CD_{\leq n})$ by the following construction. Take a numbered chord diagram $X$, and label the segments of its boundary $I_1,...,I_{2k}$, where $k$ is the number of chords. Then, we let $a_1,...,a_{2k}$ be the numbers of the numbered chord diagram corresponding to each of those segments. Then, we define $X(b_1,...,b_{2k})$ to be the chord diagram obtained by adding in $b_i$ small counterclockwise pointing isolated chords into the interval $I_i$ for all $i$. If $X(b_1,...,b_{2k})$ has more than $n$ chords, we set it to zero. Finally, we define the map $v: Ab(NCD_{\leq n})\to Ab(CD_{\leq n})$ by the formula. $$v(X) = \sum_{b_1= 0}^n\sum_{b_2= 0}^n...\sum_{b_{2k} = 0}^n \left(\prod_{i = 1}^{2k}f(a_i,b_i)\right) X(b_1,...,b_{2k}) $$

where $f$ is defined recursively by the following requirements.
\begin{itemize}
\item[1)] We have $f(0,b) = C_b$, where $C_b$ is the $b$-th Catalan number.
\item[2)] We have $f(a,0) = 1$ for all $a$.
\item[3)] If $b>0$, then $f(1,b) = 0$.
\item[4)] If $a>1$ and $b>0$, then $f(a,b) = f(a-1,b) - f(a-2,b-1)$.
\end{itemize}

Next, we let $R$ denote the subgroup of $ Ab(NCD_{\leq n})$ generated by the relations of type 1, 2, and 3. Thus, $Ab(NCD_{\leq n})/R = U_n(NCD)$. We define $U_n(CD)$ to be $Ab(CD_{\leq n})/v(R)$. Thus, we have an induced map $v: U_n(NCD)\to U_n(CD)$.

\begin{lem}\label{reduce}
Let $X$ be a braided Gauss diagram of braid index $m$, and let $I$ be an interval between chord endpoints of $X$ such that no $m$-th roots of unity are in $I$. Then, let $X_{i,j}$ denote the braided chord diagram of braid index $m+i+j$ obtained from $X$ by adding $j$ units of distance to $I$ and then stabilizing $i$ times in $I$. We claim that $$ X_{0,i} = \sum_{j = 0}^n f(i,j)X_{j,j+1} $$ in $U_{n}(VTK)$.

Note that $X_{j,j+1}$ is locally unitary in $I$, because we can position the isolated chords from the stabilizations with exactly one root of unity separating each one, as well as one root of unity on either side separating them from the endpoints of $I$.
\end{lem}

\begin{proof}
We proceed by induction. For the base case $i = 0$. We see that type 2 relations in $U_n^+(VTK)$ give us $X_{i,j} = X_{i+1,j} + X_{i,j+1}$ for all $i$ and $j$. If we repeatedly apply this relation starting at $X_{0,0}$ and stopping when any given term becomes $X_{j,j+1}$, then we see that the resulting coefficient for $X_{j,j+1}$ will be the number of paths in $\Z^2$, starting at $(0,0)$ and ending at $(j,j)$, always going either up or right, and never going above the diagonal. This is one common definition for the Catalan numbers.
For the case $i = 1$, the equation is trivially true because the sum only has one nonzero term, which is $X_{0,1}$. 

For the inductive step, suppose $i>1$, and suppose the equation is true for all smaller $i$. Then, a type 2 relation gives us $X_{0,i} = X_{0,i-1} + X_{1,i-1}$. Now, let $Y= X_{1,1}$.  Then, position the stabilization in $X_{1,1}$ so that there are no $(m+2)$-th roots of unity between the stabilization and the left endpoint of $I$. Then, using this interval to define $Y_{i,j}$ as we did for $X_{i.j}$, we have that $X_{i,j} = Y_{i-1,j-1}$. Therefore, we have $X_{0,i} = X_{0,i-1} + X_{1,i-1} = X_{0,i-1} + Y_{0,i-2}$. From our inductive assumption, we can now apply the formula to the two latter terms in this equation. We have $$ X_{0,i} = \sum_{j = 0}^n f(i-1,j)X_{j,j+1} +   f(i-2,j)Y_{j,j+1}$$ 
$$ = \sum_{j = 0}^n f(i-1,j)X_{j,j+1} +   f(i-2,j-1)X_{j,j+1} = \sum_{j = 0}^n f(i,j)X_{j,j+1} $$ 
Which proves the lemma.
\end{proof}

\begin{lem}
Earlier we defined $u: Ab(CD_{\geq 0})\to U_n^+(VTK)$, which takes chord diagrams to their unitary representatives. We claim that $u(v(r)) = 0$ for any $r\in R$, and thus there is an indued map $u: U_n(CD)\to U_n^+(VTK)$.
\end{lem}

\begin{proof}
We will check each type of relation, and verify that they indeed map to zero in $U_n^+(VTK)$. 

Using Lemma \ref{reduce}, is easy to see that type 1 relations in $R$ map to zero under $uv$. The designated intervals that have numbering zero in figure \ref{type1} will map under $v$ to linear combinations with Catalan number coefficients, which are then equivalent in $U_n^+(VTK)$ to sections of the boundary where the chords are close to each other. This turns the type $1$ relations of $R$ into linear combinations of type 1 relations of $U_n^+(VTK)$. 

Next, we check the type 2 relations of $R$. We claim these map to zero under $v$. We see that it suffices to show that the recursively defined function $f$ from the definition of $v$ has the following property for all triples of nonnegative integers $(a_1,a_2,b)$. $$ f(a_1+a_2,b+1) = f(a_1+a_2+1,b+1) + \sum_{b_1 + b_2 = b} f(a_1,b_1)f(a_2,b_2)  $$ 

First, note that when $a_1 = 1$ this formula reduces to the recurrence relation for $f$, so it suffices to prove $$ \sum_{b_1 + b_2 = b} f(a_1,b_1)f(a_2+1,b_2) = \sum_{b_1 + b_2 = b} f(a_1+1,b_1)f(a_2,b_2) $$ for all triples of nonnegative integers $(a_1,a_2,b)$. To prove this, we use induction on $b$. Our base case is $b = 0$, in which case all the terms are $1$, so the equation is true. For our inductive step, suppose we know this equation is true for all $b < b_0$. We wish to prove it for $b_0$. We have 
$$ \sum_{b_1 + b_2 = b_0} f(a_1,b_1)f(a_2+1,b_2) - \sum_{b_1 + b_2 = b_0} f(a_1+1,b_1)f(a_2,b_2) $$
$$ =   \sum_{b_1 + b_2 = b_0} f(a_1,b_1)(f(a_2,b_2)- f(a_2-1,b_2-1)) - \sum_{b_1 + b_2 = b_0} (f(a_1,b_1)- f(a_1-1,b_1-1))f(a_2,b_2)  $$
$$ = \sum_{b_1 + b_2 = b_0}f(a_1-1,b_1-1)f(a_2,b_2) -  \sum_{b_1 + b_2 = b_0} f(a_1,b_1)f(a_2-1,b_2-1)  $$
$$ = \sum_{b_1 + b_2 = b_0-1}f(a_1-1,b_1)f(a_2,b_2) -  \sum_{b_1 + b_2 = b_0-1} f(a_1,b_1)f(a_2-1,b_2)   =  0 $$
which completes the inductive step. 

Finally, we wish to prove that type 3 relations $r\in R$ are mapped to zero under $uv$. By Lemma \ref{reduce}, we see that $uv(r)$ can be transformed into a relation where we move a $m$-th root of unity through both the front and back endpoints of a chord, where $m$ is the braid index. This is a valid transformation of braided chord diagrams, simply corresponding to a homotopy of the diagram. Thus, $uv(r)= 0$.
\end{proof}

\begin{lem}
We have defined the following three maps. $$u:U_n(CD)\to U_n^+(VTK),\;\;\;\;w: U_n^+(VTK)\to U_n(NCD),\;\;\;\;v: U_n(NCD)\to U_n(CD)$$ We claim that all three of these maps are isomorphisms, and the compositions $vwu,$ $uvw$, and $wuv$, are all identity maps.
\end{lem}

\begin{proof}
$v: U_n(NCD)\to U_n(CD)$ is an isomorphism because $U_n(CD)$ was defined as the quotient $Ab(CD_{\leq n})/v(R)$, and the map $v: Ab(NCD_{\leq n}) \to Ab(CD_{\leq n})$ is surjective since it takes numbered chord diagrams where all the numbers are one to their corresponding chord diagrams.

Thus, we only have to check that the compositions $vwu$ and $uvw$ are identities. The map $vwu$ can easily be seen to be an identity because it takes a chord diagram to its unitary representative, which then maps to a numbered chord diagram where all the numbers are one, which then maps back to the original diagram. Finally, $uvw$ is the identity by Lemma \ref{reduce}. We map a braided chord diagram to the numbered chord diagram that counts roots of unity in each boundary section, which then maps to a linear combination of unitary diagrams which can be reduced by the relations of $U_n^+(VTK)$ to the original diagram.
\end{proof}

We can now prove Theorem \ref{surj}. 
\begin{proof}[Proof of Theorem \ref{surj}.]
We know that the map $Ab(CD_{\leq n})\to U_n(VTK)$ is surjective because we have proven that the induced map $U_n(CD)\to U_n(VTK)$ is an isomorphism. 
\end{proof}

Thus, we finally know that $U_n(VTK)$ is finitely generated, because it is isomorphic to the finitely generated abelian group $U_n(CD)$. However, if we actually want to compute this abelian group, we also need a finite set of relations. The problem is that the relations of $U_n(CD)$ are defined to be $v(R)$ where $R$ is the infinite dimensional space of relations in $U_n(NCD)$. We need to find a finite dimensional subspace $\tilde{R}\subseteq R$ so that $v(R) = v(\tilde{R})$. One obvious such choice is the following.

\begin{dfn}
Let $\tilde{R}$ be the subspace of $R$ generated by the following subset of relations.
\begin{itemize}
\item[1)] Type 1 relations of $R$ for which all numbers in the diagram are $1$, except where they are specified to be zero in Figure \ref{type1}.
\item[2)] Type 3 relations of $R$ for which at most two of the numbers in the diagrams are $0$, and the rest are $1$. 
\end{itemize}
\end{dfn}

Thus, $Ab(CD_{\leq n})/v(\tilde{R})$ gives us a finite presentation for the abelian group $U_n(VTK)$. It is this presentation that we used in our computer program to compute $U_n(VTK)\otimes(\Z/p\Z)$.

\newpage
\nocite{*}
\bibliography{Refrences}{}
\bibliographystyle{plain}

\end{document}